\documentclass[10pt]{article}
\usepackage{amsmath}
\usepackage{amssymb}
\usepackage{amsthm}
\usepackage{mathrsfs}
\numberwithin{equation}{section}

\begin{document}

\title{Estimates for vector-valued intrinsic square functions and their commutators on certain weighted amalgam spaces}
\author{Hua Wang \footnote{E-mail address: wanghua@pku.edu.cn.}\\
\footnotesize{College of Mathematics and Econometrics, Hunan University, Changsha 410082, P. R. China}\\
\footnotesize{\&~Department of Mathematics and Statistics, Memorial University, St. John's, NL A1C 5S7, Canada}}
\date{}
\maketitle

\begin{abstract}
In this paper, we first introduce some new kinds of weighted amalgam spaces. Then we deal with the vector-valued intrinsic square functions, which are given by
\begin{equation*}
\mathcal S_\gamma(\vec{f})(x)=\Bigg(\sum_{j=1}^\infty\big|\mathcal S_\gamma(f_j)(x)\big|^2\Bigg)^{1/2},
\end{equation*}
where $0<\gamma\leq1$ and
\begin{equation*}
\mathcal S_{\gamma}(f_j)(x)=\left(\iint_{\Gamma(x)}\Big[\sup_{\varphi\in{\mathcal C}_\gamma}\big|\varphi_t*f_j(y)\big|\Big]^2\frac{dydt}{t^{n+1}}\right)^{1/2},\quad j=1,2,\dots.
\end{equation*}
In his fundamental work, Wilson established strong-type and weak-type estimates for vector-valued intrinsic square functions on weighted Lebesgue spaces. The goal of this paper is to extend his results to these weighted amalgam spaces. Moreover, we define vector-valued analogues of commutators with $BMO(\mathbb R^n)$ functions, and obtain the mapping properties of vector-valued commutators on the weighted amalgam spaces as well. In the endpoint case, we also establish the weighted weak $L\log L$-type estimates for vector-valued commutators in the setting of weighted Lebesgue spaces.\\
MSC(2010): 42B25; 42B35; 46E30; 47B47\\
Keywords: Vector-valued intrinsic square functions; weighted amalgam spaces; vector-valued commutators; Muckenhoupt weights; Orlicz spaces.
\end{abstract}

\section{Introduction}

\newtheorem{defn}{Definition}[section]

\newtheorem{theorem}{Theorem}[section]

\newtheorem{corollary}{Corollary}[section]

\newtheorem{lemma}{Lemma}[section]

The intrinsic square functions were first introduced by Wilson in \cite{wilson1,wilson2}; they are defined as follows. For $0<\gamma\le1$, let ${\mathcal C}_\gamma$ be the family of functions $\varphi:\mathbb R^n\longmapsto\mathbb R$ such that $\varphi$'s support is contained in $\{x\in\mathbb R^n: |x|\le1\}$, $\int_{\mathbb R^n}\varphi(x)\,dx=0$, and for all $x, x'\in \mathbb R^n$,
\begin{equation*}
\big|\varphi(x)-\varphi(x')\big|\leq \big|x-x'\big|^{\gamma}.
\end{equation*}
For $(y,t)\in {\mathbb R}^{n+1}_{+}=\mathbb R^n\times(0,+\infty)$ and $f\in L^1_{{loc}}(\mathbb R^n)$, we define
\begin{equation*}
\mathcal A_\gamma(f)(y,t):=\sup_{\varphi\in{\mathcal C}_\gamma}\big|\varphi_t*f(y)\big|=\sup_{\varphi\in{\mathcal C}_\gamma}\bigg|\int_{\mathbb R^n}\varphi_t(y-z)f(z)\,dz\bigg|,
\end{equation*}
where $\varphi_t$ denotes the usual $L^1$ dilation of $\varphi:\varphi_t(y)=t^{-n}\varphi(y/t)$. Then we define the intrinsic square function of $f$ (of order $\gamma$) by the formula
\begin{equation}
\mathcal S_{\gamma}(f)(x):=\left(\iint_{\Gamma(x)}\Big[\mathcal A_\gamma(f)(y,t)\Big]^2\frac{dydt}{t^{n+1}}\right)^{1/2},
\end{equation}
where $\Gamma(x)$ denotes the usual cone of aperture one:
\begin{equation*}
\Gamma(x):=\big\{(y,t)\in{\mathbb R}^{n+1}_+:|x-y|<t\big\}.
\end{equation*}
This new function is independent of any particular kernel, and it dominates pointwise the classical square function~(Lusin area integral) and its real-variable generalizations, one can see more details in \cite{wilson1,wilson2}. In this paper, we will consider the vector-valued extension of the scalar operator $\mathcal S_{\gamma}$. Let $\vec{f}=(f_1,f_2,\ldots)$ be a sequence of locally integrable functions on $\mathbb R^n$. For any $x\in\mathbb R^n$ and $0<\gamma\leq1$, Wilson \cite{wilson2} also defined the following vector-valued intrinsic square function of $\vec{f}$ by
\begin{equation}\label{vectorvalued}
\mathcal S_\gamma(\vec{f})(x):=\bigg(\sum_{j=1}^\infty\big|\mathcal S_\gamma(f_j)(x)\big|^2\bigg)^{1/2}.
\end{equation}

In \cite{wilson2}, Wilson has established the following two results.

\begin{theorem}[\cite{wilson2}]\label{strong}
Let $0<\gamma\leq 1$, $1<p<\infty$ and $w\in A_p$(\mbox{Muckenhoupt weight class}). Then there exists a constant $C>0$ independent of $\vec{f}=(f_1,f_2,\ldots)$ such that
\begin{equation*}
\bigg\|\bigg(\sum_{j=1}^\infty\big|\mathcal S_\gamma(f_j)\big|^2\bigg)^{1/2}\bigg\|_{L^p_w}
\leq C \bigg\|\bigg(\sum_{j=1}^\infty\big|f_j\big|^2\bigg)^{1/2}\bigg\|_{L^p_w}.
\end{equation*}
\end{theorem}

\begin{theorem}[\cite{wilson2}]
Let $0<\gamma\leq 1$ and $p=1$. Then for any given weight $w$ and $\lambda>0$, there exists a constant $C>0$ independent of $\vec{f}=(f_1,f_2,\ldots)$ and $\lambda>0$ such that
\begin{equation*}
w\bigg(\bigg\{x\in\mathbb R^n:\bigg(\sum_{j=1}^\infty\big|\mathcal S_{\gamma}(f_j)(x)\big|^2\bigg)^{1/2}>\lambda\bigg\}\bigg)
\leq \frac{C}{\lambda}\int_{\mathbb R^n}\bigg(\sum_{j=1}^\infty\big|f_j(x)\big|^2\bigg)^{1/2}M(w)(x)\,dx,
\end{equation*}
where $M$ denotes the standard Hardy--Littlewood maximal operator.
\end{theorem}

If we take $w\in A_1$, then $M(w)(x)\le C\cdot w(x)$ for a.e.$x\in\mathbb R^n$ by the definition of $A_1$ weight (see Section 2). Hence, as a straightforward consequence of Theorem 1.2, we obtain

\begin{theorem}\label{weak}
Let $0<\gamma\leq 1$, $p=1$ and $w\in A_1$. Then there exists a constant $C>0$ independent of $\vec{f}=(f_1,f_2,\ldots)$ such that
\begin{equation*}
\bigg\|\bigg(\sum_{j=1}^\infty\big|\mathcal S_{\gamma}(f_j)\big|^2\bigg)^{1/2}\bigg\|_{WL^1_w}
\leq C \bigg\|\bigg(\sum_{j=1}^\infty\big|f_j\big|^2\bigg)^{1/2}\bigg\|_{L^1_w}.
\end{equation*}
\end{theorem}

Let $b$ be a locally integrable function on $\mathbb R^n$ and $0<\gamma\leq 1$, the commutators generated by $b$ and intrinsic square function $\mathcal S_{\gamma}$ are defined by the author as follows~(see \cite{wang}).
\begin{equation}
\big[b,\mathcal S_\gamma\big](f)(x):=\left(\iint_{\Gamma(x)}\sup_{\varphi\in{\mathcal C}_\gamma}\bigg|\int_{\mathbb R^n}\big[b(x)-b(z)\big]\varphi_t(y-z)f(z)\,dz\bigg|^2\frac{dydt}{t^{n+1}}\right)^{1/2}.
\end{equation}

In this paper, we will consider the vector-valued analogues of these commutator operators. Let $\vec{f}=(f_1,f_2,\ldots)$ be a sequence of locally integrable functions on $\mathbb R^n$. For any $x\in\mathbb R^n$ and $0<\gamma\leq1$, in the same way, we can define the commutators for vector-valued intrinsic square function of $\vec{f}$ as
\begin{equation}\label{vectorvaluedc}
\big[b,\mathcal S_\gamma\big](\vec{f})(x):=\bigg(\sum_{j=1}^\infty\big|\big[b,\mathcal S_\gamma\big](f_j)(x)\big|^2\bigg)^{1/2}.
\end{equation}

We equip the $n$-dimensional Euclidean space $\mathbb R^n$ with the Euclidean norm $|\cdot|$ and the Lebesgue measure $dx$. For any $r>0$ and $y\in\mathbb R^n$, let $B(y,r)=\big\{x\in\mathbb R^n:|x-y|<r\big\}$ denote the open ball centered at $y$ with radius $r$, $B(y,r)^c$ denote its complement and $|B(y,r)|$ be the Lebesgue measure of the ball $B(y,r)$. We also use the notation $\chi_{B(y,r)}$ for the characteristic function of $B(y,r)$. Let $1\leq p,q,\alpha\leq\infty$. We define the amalgam space $(L^p,L^q)^{\alpha}(\mathbb R^n)$ of $L^p(\mathbb R^n)$ and $L^q(\mathbb R^n)$ as the set of all measurable functions $f$ satisfying $f\in L^p_{loc}(\mathbb R^n)$ and $\big\|f\big\|_{(L^p,L^q)^{\alpha}(\mathbb R^n)}<\infty$, where
\begin{equation*}
\begin{split}
\big\|f\big\|_{(L^p,L^q)^{\alpha}(\mathbb R^n)}
:=&\sup_{r>0}\left\{\int_{\mathbb R^n}\Big[\big|B(y,r)\big|^{1/{\alpha}-1/p-1/q}\big\|f\cdot\chi_{B(y,r)}\big\|_{L^p(\mathbb R^n)}\Big]^qdy\right\}^{1/q}\\
=&\sup_{r>0}\Big\|\big|B(y,r)\big|^{1/{\alpha}-1/p-1/q}\big\|f\cdot\chi_{B(y,r)}\big\|_{L^p(\mathbb R^n)}\Big\|_{L^q(\mathbb R^n)},
\end{split}
\end{equation*}
with the usual modification when $p=\infty$ or $q=\infty$. This amalgam space was originally introduced by Fofana in \cite{fofana}. As proved in \cite{fofana} the space $(L^p,L^q)^{\alpha}(\mathbb R^n)$ is non-trivial if and only if $p\leq\alpha\leq q$; thus in the remaining of this paper we will always assume that the condition $p\leq\alpha\leq q$ is fulfilled. Note that
\begin{itemize}
  \item For $1\leq p\leq\alpha\leq q\leq\infty$, one can easily see that $(L^p,L^q)^{\alpha}(\mathbb R^n)\subseteq(L^p,L^q)(\mathbb R^n)$, where $(L^p,L^q)(\mathbb R^n)$ is the Wiener amalgam space defined by (see \cite{F,holland} for more information)
\begin{equation*}
(L^p,L^q)(\mathbb R^n):=\left\{f:\big\|f\big\|_{(L^p,L^q)(\mathbb R^n)}
=\left(\int_{\mathbb R^n}\Big[\big\|f\cdot\chi_{B(y,1)}\big\|_{L^p(\mathbb R^n)}\Big]^qdy\right)^{1/q}<\infty\right\};
\end{equation*}
  \item If $1\leq p<\alpha$ and $q=\infty$, then $(L^p,L^q)^{\alpha}(\mathbb R^n)$ is just the classical Morrey space $\mathcal L^{p,\kappa}(\mathbb R^n)$ defined by (with $\kappa=1-p/{\alpha}$, see \cite{morrey})
\begin{equation*}
\mathcal L^{p,\kappa}(\mathbb R^n):=\left\{f:\big\|f\big\|_{\mathcal L^{p,\kappa}(\mathbb R^n)}
=\sup_{y\in\mathbb R^n,r>0}\left(\frac{1}{|B(y,r)|^\kappa}\int_{B(y,r)}|f(x)|^p\,dx\right)^{1/p}<\infty\right\};
\end{equation*}
  \item If $p=\alpha$ and $q=\infty$, then $(L^p,L^q)^{\alpha}(\mathbb R^n)$ reduces to the usual Lebesgue space $L^{p}(\mathbb R^n)$.
\end{itemize}

In \cite{feuto2} (see also \cite{feuto1,feuto3}), Feuto considered a weighted version of the amalgam space $(L^p,L^q)^{\alpha}(w)$. A weight is any positive measurable function $w$ which is locally integrable on $\mathbb R^n$. Let $1\leq p\leq\alpha\leq q\leq\infty$ and $w$ be a weight on $\mathbb R^n$. We denote by $(L^p,L^q)^{\alpha}(w)$ the weighted amalgam space, the space of all locally integrable functions $f$ satisfying $\big\|f\big\|_{(L^p,L^q)^{\alpha}(w)}<\infty$, where
\begin{equation}\label{A}
\begin{split}
\big\|f\big\|_{(L^p,L^q)^{\alpha}(w)}
:=&\sup_{r>0}\left\{\int_{\mathbb R^n}\Big[w(B(y,r))^{1/{\alpha}-1/p-1/q}\big\|f\cdot\chi_{B(y,r)}\big\|_{L^p_w}\Big]^qdy\right\}^{1/q}\\
=&\sup_{r>0}\Big\|w(B(y,r))^{1/{\alpha}-1/p-1/q}\big\|f\cdot\chi_{B(y,r)}\big\|_{L^p_w}\Big\|_{L^q(\mathbb R^n)},
\end{split}
\end{equation}
with the usual modification when $q=\infty$ and $w(B(y,r))=\int_{B(y,r)}w(x)\,dx$ is the weighted measure of $B(y,r)$. Then for $1\leq p\leq\alpha\leq q\leq\infty$, we know that $(L^p,L^q)^{\alpha}(w)$ becomes a Banach function space with respect to the norm $\|\cdot\|_{(L^p,L^q)^{\alpha}(w)}$. Furthermore, we denote by $(WL^p,L^q)^{\alpha}(w)$ the weighted weak amalgam space of all measurable functions $f$ for which (see \cite{feuto2})
\begin{equation}\label{WA}
\begin{split}
\big\|f\big\|_{(WL^p,L^q)^{\alpha}(w)}
:=&\sup_{r>0}\left\{\int_{\mathbb R^n}\Big[w(B(y,r))^{1/{\alpha}-1/p-1/q}\big\|f\cdot\chi_{B(y,r)}\big\|_{WL^p_w}\Big]^qdy\right\}^{1/q}\\
=&\sup_{r>0}\Big\|w(B(y,r))^{1/{\alpha}-1/p-1/q}\big\|f\cdot\chi_{B(y,r)}\big\|_{WL^p_w}\Big\|_{L^q(\mathbb R^n)}<\infty.
\end{split}
\end{equation}

Note that
\begin{itemize}
  \item If $1\leq p<\alpha$ and $q=\infty$, then $(L^p,L^q)^{\alpha}(w)$ is just the weighted Morrey space $\mathcal L^{p,\kappa}(w)$ defined by (with $\kappa=1-p/{\alpha}$, see \cite{komori})
\begin{equation*}
\begin{split}
&\mathcal L^{p,\kappa}(w)\\
:=&\left\{f :\big\|f\big\|_{\mathcal L^{p,\kappa}(w)}
=\sup_{y\in\mathbb R^n,r>0}\left(\frac{1}{w(B(y,r))^{\kappa}}\int_{B(y,r)}|f(x)|^pw(x)\,dx\right)^{1/p}<\infty\right\},
\end{split}
\end{equation*}
and $(WL^p,L^q)^{\alpha}(w)$ is just the weighted weak Morrey space $W\mathcal L^{p,\kappa}(w)$ defined by (with $\kappa=1-p/{\alpha}$)
\begin{equation*}
\begin{split}
&W\mathcal L^{p,\kappa}(w)\\
:=&\left\{f :\big\|f\big\|_{W\mathcal L^{p,\kappa}(w)}
=\sup_{y\in\mathbb R^n,r>0}\sup_{\lambda>0}\frac{1}{w(B(y,r))^{\kappa/p}}\lambda\cdot\Big[w\big(\big\{x\in B(y,r):|f(x)|>\lambda\big\}\big)\Big]^{1/p}
<\infty\right\};
\end{split}
\end{equation*}
  \item If $p=\alpha$ and $q=\infty$, then $(L^p,L^q)^{\alpha}(w)$ reduces to the weighted Lebesgue space $L^{p}_w(\mathbb R^n)$, and $(WL^p,L^q)^{\alpha}(w)$ reduces to the weighted weak Lebesgue space $WL^{p}_w(\mathbb R^n)$.
\end{itemize}

The main purpose of this paper is twofold. We first define some new kinds of weighted amalgam spaces, and then we are going to prove that vector-valued intrinsic square functions \eqref{vectorvalued} and associated vector-valued commutators \eqref{vectorvaluedc} which are known to be bounded on weighted Lebesgue spaces, are also bounded on these new weighted
spaces under appropriate conditions.

Throughout this paper, the letter $C$ always denotes a positive constant independent of the main parameters involved, but it may be different from line to line. We also use $A\approx B$ to denote the equivalence of $A$ and $B$; that is, there exist two positive constants $C_1$, $C_2$ independent of $A$ and $B$ such that $C_1 A\leq B\leq C_2 A$.

\section{Main results}

\subsection{Notation and preliminaries}
A weight $w$ is said to belong to the Muckenhoupt's class $A_p$ for $1<p<\infty$, if there exists a constant $C>0$ such that
\begin{equation*}
\left(\frac1{|B|}\int_B w(x)\,dx\right)^{1/p}\left(\frac1{|B|}\int_B w(x)^{-p'/p}\,dx\right)^{1/{p'}}\leq C
\end{equation*}
for every ball $B\subset\mathbb R^n$, where $p'$ is the dual of $p$ such that $1/p+1/{p'}=1$. The class $A_1$ is defined replacing the above inequality by
\begin{equation*}
\frac1{|B|}\int_B w(x)\,dx\leq C\cdot\underset{x\in B}{\mbox{ess\,inf}}\;w(x)
\end{equation*}
for every ball $B\subset\mathbb R^n$. We also define $A_\infty=\bigcup_{1\leq p<\infty}A_p$. For some $t>0$, the notation $t B$ stands for the ball with the same center as $B=B(y,r_B)$ and with radius $t\cdot r_B$. It is well known that if $w\in A_p$ with $1\leq p<\infty$(or $w\in A_\infty$), then $w$ satisfies the \emph{doubling} condition; that is, for any ball $B$ in $\mathbb R^n$, there exists an absolute constant $C>0$ such that (see \cite{garcia})
\begin{equation}\label{weights}
w(2B)\leq C\,w(B).
\end{equation}
When $w$ satisfies this \emph{doubling} condition \eqref{weights}, we denote $w\in\Delta_2$ for brevity.
In general, for $w\in A_1$ and any $l\in\mathbb Z_+$, there exists an absolute constant $C>0$ such that (see \cite{garcia})
\begin{equation}\label{general weights}
w\big(2^l B\big)\le C\cdot 2^{ln}w(B).
\end{equation}
Moreover, if $w\in A_\infty$, then for any ball $B$ in $\mathbb R^n$ and any measurable subset $E$ of a ball $B$, there exists a number $\delta>0$ independent of $E$ and $B$ such that (see \cite{garcia})
\begin{equation}\label{compare}
\frac{w(E)}{w(B)}\le C\left(\frac{|E|}{|B|}\right)^\delta.
\end{equation}
Equivalently, we could define the above notions with cubes
instead of balls. Hence we shall use these two different definitions appropriate to calculations.
Given a weight $w$ on $\mathbb R^n$, as usual, the weighted Lebesgue space $L^p_w(\mathbb R^n)$ for $1\leq p<\infty$ is defined as the set of all functions $f$ such that
\begin{equation*}
\big\|f\big\|_{L^p_w}:=\bigg(\int_{\mathbb R^n}\big|f(x)\big|^pw(x)\,dx\bigg)^{1/p}<\infty.
\end{equation*}
We also denote by $WL^p_w(\mathbb R^n)$($1\leq p<\infty$) the weighted weak Lebesgue space consisting of all measurable functions $f$ such that
\begin{equation*}
\big\|f\big\|_{WL^p_w}:=
\sup_{\lambda>0}\lambda\cdot\Big[w\big(\big\{x\in\mathbb R^n:|f(x)|>\lambda\big\}\big)\Big]^{1/p}<\infty.
\end{equation*}

We next recall some basic definitions and facts about Orlicz spaces needed for the proofs of the main results. For further information on the subject, one can see \cite{rao}. A function $\mathcal A$ is called a Young function if it is continuous, nonnegative, convex and strictly increasing on $[0,+\infty)$ with $\mathcal A(0)=0$ and $\mathcal A(t)\to +\infty$ as $t\to +\infty$. Given a Young function $\mathcal A$, we define the $\mathcal A$-average of a function $f$ over a ball $B$ by means of the following Luxemburg norm:
\begin{equation*}
\big\|f\big\|_{\mathcal A,B}
:=\inf\left\{\lambda>0:\frac{1}{|B|}\int_B\mathcal A\left(\frac{|f(x)|}{\lambda}\right)dx\leq1\right\}.
\end{equation*}
When $\mathcal A(t)=t^p$, $1\leq p<\infty$, it is easy to see that
\begin{equation*}
\big\|f\big\|_{\mathcal A,B}=\left(\frac{1}{|B|}\int_B\big|f(x)\big|^p\,dx\right)^{1/p};
\end{equation*}
that is, the Luxemburg norm coincides with the normalized $L^p$ norm. Given a Young function $\mathcal A$, we use $\bar{\mathcal A}$ to denote the complementary Young function associated to $\mathcal A$. Then the following generalized H\"older's inequality holds for any given ball $B$:
\begin{equation*}
\frac{1}{|B|}\int_B\big|f(x)\cdot g(x)\big|\,dx\leq 2\big\|f\big\|_{\mathcal A,B}\big\|g\big\|_{\bar{\mathcal A},B}.
\end{equation*}
In particular, when $\mathcal A(t)=t\cdot(1+\log^+t)$, we know that its complementary Young function is $\bar{\mathcal A}(t)\approx\exp(t)-1$. In this situation, we denote
\begin{equation*}
\big\|f\big\|_{L\log L,B}=\big\|f\big\|_{\mathcal A,B}, \qquad
\big\|g\big\|_{\exp L,B}=\big\|g\big\|_{\bar{\mathcal A},B}.
\end{equation*}
So we have
\begin{equation}\label{holder}
\frac{1}{|B|}\int_B\big|f(x)\cdot g(x)\big|\,dx\leq 2\big\|f\big\|_{L\log L,B}\big\|g\big\|_{\exp L,B}.
\end{equation}

\subsection{Weighted amalgam spaces}

Let us begin with the definitions of the weighted amalgam spaces with Lebesgue measure in (\ref{A}) and (\ref{WA}) replaced by weighted measure.

\begin{defn}\label{amalgam}
Let $1\leq p\leq\alpha\leq q\leq\infty$, and let $w,\mu$ be two weights on $\mathbb R^n$. We denote by $(L^p,L^q)^{\alpha}(w;\mu)$ the weighted amalgam space, the space of all locally integrable functions $f$ with finite norm
\begin{equation*}
\begin{split}
\big\|f\big\|_{(L^p,L^q)^{\alpha}(w;\mu)}
:=&\sup_{r>0}\left\{\int_{\mathbb R^n}\Big[w(B(y,r))^{1/{\alpha}-1/p-1/q}\big\|f\cdot\chi_{B(y,r)}\big\|_{L^p_w}\Big]^q\mu(y)\,dy\right\}^{1/q}\\
=&\sup_{r>0}\Big\|w(B(y,r))^{1/{\alpha}-1/p-1/q}\big\|f\cdot\chi_{B(y,r)}\big\|_{L^p_w}\Big\|_{L^q_{\mu}}<\infty,
\end{split}
\end{equation*}
with the usual modification when $q=\infty$. Then we can see that the space $(L^p,L^q)^{\alpha}(w;\mu)$ equipped with the norm $\big\|\cdot\big\|_{(L^p,L^q)^{\alpha}(w;\mu)}$ is a Banach function space. Furthermore, we denote by $(WL^p,L^q)^{\alpha}(w;\mu)$ the weighted weak amalgam space of all measurable functions $f$ for which
\begin{equation*}
\begin{split}
\big\|f\big\|_{(WL^p,L^q)^{\alpha}(w;\mu)}
:=&\sup_{r>0}\left\{\int_{\mathbb R^n}\Big[w(B(y,r))^{1/{\alpha}-1/p-1/q}\big\|f\cdot\chi_{B(y,r)}\big\|_{WL^p_w}\Big]^q\mu(y)\,dy\right\}^{1/q}\\
=&\sup_{r>0}\Big\|w(B(y,r))^{1/{\alpha}-1/p-1/q}\big\|f\cdot\chi_{B(y,r)}\big\|_{WL^p_w}\Big\|_{L^q_{\mu}}<\infty,
\end{split}
\end{equation*}
with the usual modification when $q=\infty$.
\end{defn}

Recently, in \cite{wang2}, we have established the strong type and weak type estimates for vector-valued intrinsic square functions on some Morrey-type spaces. Inspired by the works mentioned above, it is natural to discuss the boundedness properties in the context of weighted amalgam spaces. We will show that vector-valued intrinsic square function is bounded on $(L^p,L^q)^{\alpha}(w;\mu)$, and is bounded from $(L^1,L^q)^{\alpha}(w;\mu)$ into $(WL^1,L^q)^{\alpha}(w;\mu)$. Our first two results in this paper can be formulated as follows.

\begin{theorem}\label{mainthm:1}
Let $0<\gamma\le1$. Assume that $1<p\leq\alpha<q\leq\infty$, $w\in A_p$ and $\mu\in\Delta_2$. Then there is a constant $C>0$ independent of $\vec{f}=(f_1,f_2,\ldots)$ such that
\begin{equation*}
\bigg\|\bigg(\sum_{j=1}^\infty\big|\mathcal S_{\gamma}(f_j)\big|^2\bigg)^{1/2}\bigg\|_{(L^p,L^q)^{\alpha}(w;\mu)}
\leq C \bigg\|\bigg(\sum_{j=1}^\infty\big|f_j\big|^2\bigg)^{1/2}\bigg\|_{(L^p,L^q)^{\alpha}(w;\mu)}.
\end{equation*}
\end{theorem}

\begin{theorem}\label{mainthm:2}
Let $0<\gamma\le1$. Assume that $p=1$, $1\leq\alpha<q\leq\infty$, $w\in A_1$ and $\mu\in\Delta_2$. Then there is a constant $C>0$ independent of $\vec{f}=(f_1,f_2,\ldots)$ such that
\begin{equation*}
\bigg\|\bigg(\sum_{j=1}^\infty\big|\mathcal S_{\gamma}(f_j)\big|^2\bigg)^{1/2}\bigg\|_{(WL^1,L^q)^{\alpha}(w;\mu)}
\leq C\bigg\|\bigg(\sum_{j=1}^\infty\big|f_j\big|^2\bigg)^{1/2}\bigg\|_{(L^1,L^q)^{\alpha}(w;\mu)}.
\end{equation*}
\end{theorem}

For the strong type estimate of vector-valued commutator (\ref{vectorvaluedc}) defined above on the weighted amalgam spaces, we will prove

\begin{theorem}\label{mainthm:3}
Let $0<\gamma\le1$ and $b\in BMO(\mathbb R^n)$. Assume that $1<p\leq\alpha<q\leq\infty$, $w\in A_p$ and $\mu\in\Delta_2$. Then there is a constant $C>0$ independent of $\vec{f}=(f_1,f_2,\ldots)$ such that
\begin{equation*}
\bigg\|\bigg(\sum_{j=1}^\infty\big|\big[b,\mathcal S_\gamma\big](f_j)\big|^2\bigg)^{1/2}\bigg\|_{(L^p,L^q)^{\alpha}(w;\mu)}
\leq C \bigg\|\bigg(\sum_{j=1}^\infty\big|f_j\big|^2\bigg)^{1/2}\bigg\|_{(L^p,L^q)^{\alpha}(w;\mu)}.
\end{equation*}
\end{theorem}

To obtain endpoint estimate for the vector-valued commutator (\ref{vectorvaluedc}), we first need to define the weighted $\mathcal A$-average of a function $f$ over a ball $B$ by means of the weighted Luxemburg norm; that is, given a Young function $\mathcal A$ and $w\in A_\infty$, we define (see \cite{rao,zhang})
\begin{equation*}
\big\|f\big\|_{\mathcal A(w),B}:=\inf\left\{\sigma>0:\frac{1}{w(B)}
\int_B\mathcal A\left(\frac{|f(x)|}{\sigma}\right)\cdot w(x)\,dx\leq1\right\}.
\end{equation*}
When $\mathcal A(t)=t$, this norm is denoted by $\|\cdot\|_{L(w),B}$, when $\mathcal A(t)=t\cdot(1+\log^+t)$ and $\log^+t=\max\{\log t,0\}$, this norm is also denoted by $\|\cdot\|_{L\log L(w),B}$. The complementary Young function of $t\cdot(1+\log^+t)$ is $\exp(t)-1$ with mean Luxemburg norm denoted by $\|\cdot\|_{\exp L(w),B}$. For $w\in A_\infty$ and for every ball $B$ in $\mathbb R^n$, we can also show the weighted version of \eqref{holder}. Namely, the following generalized H\"older's inequality in the weighted setting
\begin{equation}\label{Wholder}
\frac{1}{w(B)}\int_B|f(x)\cdot g(x)|w(x)\,dx\leq C\big\|f\big\|_{L\log L(w),B}\big\|g\big\|_{\exp L(w),B}
\end{equation}
is valid (see \cite{zhang} for instance). Now we introduce new weighted spaces of $L\log L$ type as follows.

\begin{defn}
Let $p=1$, $1\leq\alpha\leq q\leq\infty$, and let $w,\mu$ be two weights on $\mathbb R^n$. We denote by $(L\log L,L^q)^{\alpha}(w;\mu)$ the weighted amalgam space of $L\log L$ type, the space of all locally integrable functions $f$ defined on $\mathbb R^n$ with finite norm $\big\|f\big\|_{(L\log L,L^q)^{\alpha}(w;\mu)}$.
\begin{equation*}
(L\log L,L^q)^{\alpha}(w;\mu):=\left\{f\in L^1_{loc}(w):\big\|f\big\|_{(L\log L,L^q)^{\alpha}(w;\mu)}<\infty\right\},
\end{equation*}
where
\begin{equation*}
\begin{split}
\big\|f\big\|_{(L\log L,L^q)^{\alpha}(w;\mu)}
:=&\sup_{r>0}\left\{\int_{\mathbb R^n}\Big[w(B(y,r))^{1/{\alpha}-1/q}\big\|f\big\|_{L\log L(w),B(y,r)}\Big]^q\mu(y)\,dy\right\}^{1/q}\\
=&\sup_{r>0}\Big\|w(B(y,r))^{1/{\alpha}-1/q}\big\|f\big\|_{L\log L(w),B(y,r)}\Big\|_{L^q_{\mu}}.
\end{split}
\end{equation*}
\end{defn}

Note that $t\leq t\cdot(1+\log^+t)$ for all $t>0$, then for any ball $B(y,r)\subset\mathbb R^n$ and $w\in A_\infty$, we have $\big\|f\big\|_{L(w),B(y,r)}\leq \big\|f\big\|_{L\log L(w),B(y,r)}$ by definition, i.e., the inequality
\begin{equation}\label{main esti1}
\big\|f\big\|_{L(w),B(y,r)}=\frac{1}{w(B(y,r))}\int_{B(y,r)}|f(x)|\cdot w(x)\,dx\leq\big\|f\big\|_{L\log L(w),B(y,r)}
\end{equation}
holds for any ball $B(y,r)\subset\mathbb R^n$.  Hence, for $1\leq\alpha\leq q\leq\infty$, we can further see the following inclusion:
\begin{equation*}
(L\log L,L^q)^{\alpha}(w;\mu)\subset (L^1,L^q)^{\alpha}(w;\mu).
\end{equation*}
For the endpoint estimate of commutators generated by $BMO(\mathbb R^n)$ function and vector-valued intrinsic square function, we will also prove the following weak-type $L\log L$ inequality in the context of weighted amalgam spaces.

\begin{theorem}\label{mainthm:4}
Let $0<\gamma\leq1$ and $b\in BMO(\mathbb R^n)$. Assume that $p=1$, $1\leq\alpha<q\leq\infty$, $w\in A_1$ and $\mu\in\Delta_2$, then for any given $\sigma>0$ and any ball $B(y,r)\subset\mathbb R^n$ with $y\in\mathbb R^n$, $r>0$, there is a constant $C>0$ independent of $\vec{f}=(f_1,f_2,\ldots)$, $B(y,r)$ and $\sigma>0$ such that
\begin{equation*}
\begin{split}
&\Bigg\|w(B(y,r))^{1/{\alpha}-1-1/q}\cdot
w\bigg(\bigg\{x\in B(y,r):\bigg(\sum_{j=1}^\infty\big|\big[b,\mathcal S_{\gamma}\big](f_j)(x)\big|^2\bigg)^{1/2}>\sigma\bigg\}\bigg)\Bigg\|_{L^q_{\mu}}\\
&\leq C\cdot\bigg\|\Phi\bigg(\frac{\|\vec{f}(\cdot)\|_{\ell^2}}{\sigma}\bigg)\bigg\|_{(L\log L,L^q)^{\alpha}(w;\mu)},
\end{split}
\end{equation*}
where
\begin{equation*}
\Phi(t)=t\cdot(1+\log^+t),\qquad \big\|\vec{f}(x)\big\|_{\ell^2}=\bigg(\sum_{j=1}^\infty|f_j(x)|^2\bigg)^{1/2},
\end{equation*}
and the norm $\|\cdot\|_{L^q_{\mu}}$ is taken with respect to the variable $y$, i.e.,
\begin{equation*}
\begin{split}
&\Bigg\|w(B(y,r))^{1/{\alpha}-1-1/q}\cdot
w\bigg(\bigg\{x\in B(y,r):\bigg(\sum_{j=1}^\infty\big|\big[b,\mathcal S_{\gamma}\big](f_j)(x)\big|^2\bigg)^{1/2}>\sigma\bigg\}\bigg)\Bigg\|_{L^q_{\mu}}\\
=&\left\{\int_{\mathbb R^n}\bigg[w(B(y,r))^{1/{\alpha}-1-1/q}\cdot w\bigg(\bigg\{x\in B(y,r):\bigg(\sum_{j=1}^\infty\big|\big[b,\mathcal S_{\gamma}\big](f_j)(x)\big|^2\bigg)^{1/2}>\sigma\bigg\}\bigg)\bigg]^q\mu(y)\,dy\right\}^{1/q}.
\end{split}
\end{equation*}
\end{theorem}

\newtheorem{rem}{Remark}[section]

\begin{rem}
From the above definitions and Theorem $\ref{mainthm:4}$, we can roughly say that the vector-valued commutator $(\ref{vectorvaluedc})$ is bounded from $(L\log L,L^q)^{\alpha}(w;\mu)$ into $(WL^1,L^q)^{\alpha}(w;\mu)$ whenever $1\leq\alpha<q\leq\infty$, $w\in A_1$ and $\mu\in\Delta_2$.
\end{rem}

\section{Proofs of Theorems \ref{mainthm:1} and \ref{mainthm:2}}
\begin{proof}[Proof of Theorem $\ref{mainthm:1}$]
Let $1<p\leq\alpha<q\leq\infty$ and $\big(\sum_{j=1}^\infty|f_j|^2\big)^{1/2}\in(L^p,L^q)^{\alpha}(w;\mu)$ with $w\in A_p$ and $\mu\in\Delta_2$. For an arbitrary point $y\in\mathbb R^n$ and $r>0$, we set $B=B(y,r)$ for the ball centered at $y$ and of radius $r$, $2B=B(y,2r)$. We represent $f_j$ as
\begin{equation*}
f_j=f_j\cdot\chi_{2B}+f_j\cdot\chi_{(2B)^c}:=f^0_j+f^\infty_j,
\end{equation*}
where $\chi_{2B}$ denotes the characteristic function of $2B=B(y,2r)$, $j=1,2,\ldots$. Then we write
\begin{align}\label{I}
&w(B(y,r))^{1/{\alpha}-1/p-1/q}\bigg\|\bigg(\sum_{j=1}^\infty\big|\mathcal S_{\gamma}(f_j)\big|^2\bigg)^{1/2}\cdot\chi_{B(y,r)}\bigg\|_{L^p_w}\notag\\
&=w(B(y,r))^{1/{\alpha}-1/p-1/q}\bigg(\int_{B(y,r)}\bigg(\sum_{j=1}^\infty\big|\mathcal S_{\gamma}(f_j)(x)\big|^2\bigg)^{p/2}w(x)\,dx\bigg)^{1/p}\notag\\
&\leq w(B(y,r))^{1/{\alpha}-1/p-1/q}\bigg(\int_{B(y,r)}\bigg(\sum_{j=1}^\infty\big|\mathcal S_{\gamma}(f^0_j)(x)\big|^2\bigg)^{p/2}w(x)\,dx\bigg)^{1/p}\notag\\
&+w(B(y,r))^{1/{\alpha}-1/p-1/q}\bigg(\int_{B(y,r)}\bigg(\sum_{j=1}^\infty\big|\mathcal S_{\gamma}(f^\infty_j)(x)\big|^2\bigg)^{p/2}w(x)\,dx\bigg)^{1/p}\notag\\
&:=I_1(y,r)+I_2(y,r).
\end{align}
Below we will give the estimates of $I_1(y,r)$ and $I_2(y,r)$, respectively. By the weighted $L^p$ boundedness of vector-valued intrinsic square function (see Theorem \ref{strong}), we have
\begin{align}\label{I1}
I_1(y,r)&\leq w(B(y,r))^{1/{\alpha}-1/p-1/q}\bigg\|\bigg(\sum_{j=1}^\infty\big|\mathcal S_{\gamma}(f^0_j)\big|^2\bigg)^{1/2}\bigg\|_{L^p_w}\notag\\
&\leq C\cdot w(B(y,r))^{1/{\alpha}-1/p-1/q}
\bigg(\int_{B(y,2r)}\bigg(\sum_{j=1}^\infty\big|f_j(x)\big|^2\bigg)^{p/2}w(x)\,dx\bigg)^{1/p}\notag\\
&=C\cdot w(B(y,2r))^{1/{\alpha}-1/p-1/q}\bigg\|\bigg(\sum_{j=1}^\infty\big|f_j\big|^2\bigg)^{1/2}\cdot\chi_{B(y,2r)}\bigg\|_{L^p_w}\notag\\
&\times\frac{w(B(y,r))^{1/{\alpha}-1/p-1/q}}{w(B(y,2r))^{1/{\alpha}-1/p-1/q}}.
\end{align}
Moreover, since $1/{\alpha}-1/p-1/q<0$ and $w\in A_p$ with $1<p<\infty$, then by doubling inequality (\ref{weights}), we obtain
\begin{equation}\label{doubling1}
\frac{w(B(y,r))^{1/{\alpha}-1/p-1/q}}{w(B(y,2r))^{1/{\alpha}-1/p-1/q}}\leq C.
\end{equation}
Substituting the above inequality \eqref{doubling1} into \eqref{I1} yields the inequality,
\begin{equation}\label{I1yr}
I_1(y,r)\leq C\cdot w(B(y,2r))^{1/{\alpha}-1/p-1/q}\bigg\|\bigg(\sum_{j=1}^\infty\big|f_j\big|^2\bigg)^{1/2}\cdot\chi_{B(y,2r)}\bigg\|_{L^p_w}.
\end{equation}
As for the second term $I_2(y,r)$, for any given $\varphi\in{\mathcal C}_\gamma$, $0<\gamma\le1$, $j=1,2,\ldots$, and $(\xi,t)\in\Gamma(x)$ with $x\in B(y,r)$, we have
\begin{align}\label{Key1}
\bigg|\int_{\mathbb R^n}\varphi_t(\xi-z)f^\infty_j(z)\,dz\bigg|
&=\bigg|\int_{B(y,2r)^c}\varphi_t(\xi-z)f_j(z)\,dz\bigg|\notag\\
&\leq C\cdot t^{-n}\int_{B(y,2r)^c\cap\{z:|\xi-z|\le t\}}\big|f_j(z)\big|\,dz\notag\\
&\leq C\cdot t^{-n}\sum_{l=1}^\infty\int_{B(y,2^{l+1}r)\backslash B(y,2^{l}r)\cap\{z:|\xi-z|\le t\}}\big|f_j(z)\big|\,dz.
\end{align}
Since $|\xi-z|\le t$ and $(\xi,t)\in\Gamma(x)$, then one has $|x-z|\leq|x-\xi|+|\xi-z|\leq 2t$. Hence, for any $x\in B(y,r)$ and $z\in B(y,2^{l+1}r)\backslash B(y,2^{l}r)$, a direct computation shows that
\begin{equation}\label{Key2}
2t\geq |x-z|\geq |z-y|-|x-y|\geq 2^{l-1}r.
\end{equation}
Therefore, by using the inequalities \eqref{Key1} and \eqref{Key2} derived above, together with Minkowski's inequality for integrals, we can deduce that
\begin{equation*}
\begin{split}
\mathcal S_\gamma(f^\infty_j)(x)&=\left(\iint_{\Gamma(x)}\sup_{\varphi\in{\mathcal C}_\gamma}\bigg|\int_{\mathbb R^n} \varphi_t(\xi-z)f^\infty_j(z)\,dz\bigg|^2\frac{d\xi dt}{t^{n+1}}\right)^{1/2}\\
&\leq C\left(\int_{2^{l-2}r}^\infty\int_{|x-\xi|<t}\bigg|t^{-n}\sum_{l=1}^\infty\int_{B(y,2^{l+1}r)\backslash B(y,2^{l}r)} \big|f_j(z)\big|\,dz\bigg|^2\frac{d\xi dt}{t^{n+1}}\right)^{1/2}\\
&\le C\left(\sum_{l=1}^\infty\int_{B(y,2^{l+1}r)\backslash B(y,2^{l}r)}\big|f_j(z)\big|\,dz\right)
\left(\int_{2^{l-2}r}^\infty\frac{dt}{t^{2n+1}}\right)^{1/2}\\
&\leq C\sum_{l=1}^\infty\frac{1}{|B(y,2^{l+1}r)|}\int_{B(y,2^{l+1}r)\backslash B(y,2^{l}r)}\big|f_j(z)\big|\,dz.
\end{split}
\end{equation*}
Then by duality and Cauchy--Schwarz inequality, we get the following pointwise estimate for any $x\in B(y,r)$.
\begin{align}\label{key estimate1}
\bigg(\sum_{j=1}^\infty\big|\mathcal S_\gamma(f^\infty_j)(x)\big|^2\bigg)^{1/2}
&\leq C\Bigg(\sum_{j=1}^\infty\bigg|\sum_{l=1}^\infty\frac{1}{|B(y,2^{l+1}r)|}
\int_{B(y,2^{l+1}r)\backslash B(y,2^{l}r)}\big|f_j(z)\big|\,dz\bigg|^2\Bigg)^{1/2}\notag\\
&=C\sup_{(\sum_{j=1}^\infty|\zeta_j|^2)^{1/2}\leq1}\sum_{j=1}^\infty\bigg(\sum_{l=1}^\infty\frac{1}{|B(y,2^{l+1}r)|}
\int_{B(y,2^{l+1}r)}\big|f_j(z)\big|\,dz\cdot\zeta_j\bigg)\notag\\
&\leq C\sum_{l=1}^\infty\frac{1}{|B(y,2^{l+1}r)|}\int_{B(y,2^{l+1}r)}\sup_{(\sum_{j=1}^\infty|\zeta_j|^2)^{1/2}\leq1}
\bigg(\sum_{j=1}^\infty\big|f_j(z)\big|\cdot\zeta_j\bigg)dz\notag\\
&\leq C\sum_{l=1}^\infty\frac{1}{|B(y,2^{l+1}r)|}\int_{B(y,2^{l+1}r)}\bigg(\sum_{j=1}^\infty\big|f_j(z)\big|^2\bigg)^{1/2}dz.
\end{align}
From this pointwise estimate, it follows that
\begin{equation*}
I_2(y,r)\leq C\cdot w(B(y,r))^{1/{\alpha}-1/q}
\sum_{l=1}^\infty\frac{1}{|B(y,2^{l+1}r)|}\int_{B(y,2^{l+1}r)}\bigg(\sum_{j=1}^\infty\big|f_j(z)\big|^2\bigg)^{1/2}dz.
\end{equation*}
Applying H\"older's inequality and $A_p$ condition on $w$, we get
\begin{equation*}
\begin{split}
&\frac{1}{|B(y,2^{l+1}r)|}\int_{B(y,2^{l+1}r)}\bigg(\sum_{j=1}^\infty\big|f_j(z)\big|^2\bigg)^{1/2}dz\\
&\leq\frac{1}{|B(y,2^{l+1}r)|}\bigg(\int_{B(y,2^{l+1}r)}\bigg(\sum_{j=1}^\infty\big|f_j(z)\big|^2\bigg)^{p/2}w(z)\,dz\bigg)^{1/p}
\left(\int_{B(y,2^{l+1}r)}w(z)^{-{p'}/p}\,dz\right)^{1/{p'}}\\
&\leq C\bigg(\int_{B(y,2^{l+1}r)}\bigg(\sum_{j=1}^\infty\big|f_j(z)\big|^2\bigg)^{p/2}w(z)\,dz\bigg)^{1/p}\cdot w\big(B(y,2^{l+1}r)\big)^{-1/p}.
\end{split}
\end{equation*}
Hence,
\begin{equation}\label{I2yr}
\begin{split}
I_2(y,r)&\leq C\cdot w(B(y,r))^{1/{\alpha}-1/q}\\
&\times\sum_{l=1}^\infty\bigg(\int_{B(y,2^{l+1}r)}\bigg(\sum_{j=1}^\infty\big|f_j(z)\big|^2\bigg)^{p/2}w(z)\,dz\bigg)^{1/p}\cdot w\big(B(y,2^{l+1}r)\big)^{-1/p}\\
&=C\sum_{l=1}^\infty w(B(y,2^{l+1}r))^{1/{\alpha}-1/p-1/q}\bigg\|\bigg(\sum_{j=1}^\infty\big|f_j\big|^2\bigg)^{1/2}\cdot\chi_{B(y,2^{l+1}r)}\bigg\|_{L^p_w}\\
&\times\frac{w(B(y,r))^{1/{\alpha}-1/q}}{w(B(y,2^{l+1}r))^{1/{\alpha}-1/q}}.
\end{split}
\end{equation}
Notice that $w\in A_p\subset A_\infty$ for $1\leq p<\infty$, then by using inequality (\ref{compare}) with exponent $\delta>0$, we can see that
\begin{align}\label{psi1}
\sum_{l=1}^\infty\frac{w(B(y,r))^{1/{\alpha}-1/q}}{w(B(y,2^{l+1}r))^{1/{\alpha}-1/q}}
&\leq C\sum_{l=1}^\infty\left(\frac{|B(y,r)|}{|B(y,2^{l+1}r)|}\right)^{\delta(1/{\alpha}-1/q)}\notag\\
&= C\sum_{l=1}^\infty\left(\frac{1}{2^{(l+1)n}}\right)^{\delta(1/{\alpha}-1/q)}\notag\\
&\leq C.
\end{align}
Here the exponent $\delta(1/{\alpha}-1/q)$ is positive by the assumption $\alpha<q$, which guarantees that the last series is convergent. Therefore by taking the $L^q_{\mu}$-norm of both sides of \eqref{I}(with respect to the variable $y$), and then using Minkowski's inequality, \eqref{I1yr}, \eqref{I2yr} and \eqref{psi1}, we obtain
\begin{equation*}
\begin{split}
&\bigg\|w(B(y,r))^{1/{\alpha}-1/p-1/q}\bigg\|\bigg(\sum_{j=1}^\infty\big|\mathcal S_{\gamma}(f_j)\big|^2\bigg)^{1/2}\cdot\chi_{B(y,r)}\bigg\|_{L^p_w}\bigg\|_{L^q_{\mu}}\\
&\leq\big\|I_1(y,r)\big\|_{L^q_{\mu}}+\big\|I_2(y,r)\big\|_{L^q_{\mu}}\\
&\leq C\bigg\|w(B(y,2r))^{1/{\alpha}-1/p-1/q}\bigg\|\bigg(\sum_{j=1}^\infty\big|\mathcal S_{\gamma}(f_j)\big|^2\bigg)^{1/2}\cdot\chi_{B(y,2r)}\bigg\|_{L^p_w}\bigg\|_{L^q_{\mu}}\\
&+C\sum_{l=1}^\infty\bigg\|w(B(y,2^{l+1}r))^{1/{\alpha}-1/p-1/q}\bigg\|\bigg(\sum_{j=1}^\infty\big|\mathcal S_{\gamma}(f_j)\big|^2\bigg)^{1/2}\cdot\chi_{B(y,2^{l+1}r)}\bigg\|_{L^p_w}\bigg\|_{L^q_{\mu}}\\
&\times\frac{w(B(y,r))^{1/{\alpha}-1/q}}{w(B(y,2^{l+1}r))^{1/{\alpha}-1/q}}\\
&\leq C\bigg\|\bigg(\sum_{j=1}^\infty\big|f_j\big|^2\bigg)^{1/2}\bigg\|_{(L^p,L^q)^{\alpha}(w;\mu)}
+C\bigg\|\bigg(\sum_{j=1}^\infty\big|f_j\big|^2\bigg)^{1/2}\bigg\|_{(L^p,L^q)^{\alpha}(w;\mu)}\\
&\times\sum_{l=1}^\infty\frac{w(B(y,r))^{1/{\alpha}-1/q}}{w(B(y,2^{l+1}r))^{1/{\alpha}-1/q}}\\
&\leq C\bigg\|\bigg(\sum_{j=1}^\infty\big|f_j\big|^2\bigg)^{1/2}\bigg\|_{(L^p,L^q)^{\alpha}(w;\mu)}.
\end{split}
\end{equation*}
Thus, by taking the supremum over all $r>0$, we complete the proof of Theorem \ref{mainthm:1}.
\end{proof}

\begin{proof}[Proof of Theorem $\ref{mainthm:2}$]
Let $p=1$, $1\leq\alpha<q\leq\infty$ and $\big(\sum_{j=1}^\infty|f_j|^2\big)^{1/2}\in(L^1,L^q)^{\alpha}(w;\mu)$ with $w\in A_1$ and $\mu\in\Delta_2$. For an arbitrary ball $B=B(y,r)\subset\mathbb R^n$ with $y\in\mathbb R^n$ and $r>0$, we represent $f_j$ as
\begin{equation*}
f_j=f_j\cdot\chi_{2B}+f_j\cdot\chi_{(2B)^c}:=f^0_j+f^\infty_j,\quad j=1,2,\ldots;
\end{equation*}
then one can write
\begin{align}\label{Iprime}
&w(B(y,r))^{1/{\alpha}-1-1/q}\bigg\|\bigg(\sum_{j=1}^\infty\big|\mathcal S_{\gamma}(f_j)\big|^2\bigg)^{1/2}\cdot\chi_{B(y,r)}\bigg\|_{WL^1_w}\notag\\
&\leq 2\cdot w(B(y,r))^{1/{\alpha}-1-1/q}\bigg\|\bigg(\sum_{j=1}^\infty\big|\mathcal S_{\gamma}(f^0_j)\big|^2\bigg)^{1/2}\cdot\chi_{B(y,r)}\bigg\|_{WL^1_w}\notag\\
&+2\cdot w(B(y,r))^{1/{\alpha}-1-1/q}\bigg\|\bigg(\sum_{j=1}^\infty\big|\mathcal S_{\gamma}(f^\infty_j)\big|^2\bigg)^{1/2}\cdot\chi_{B(y,r)}\bigg\|_{WL^1_w}\notag\\
&:=I'_1(y,r)+I'_2(y,r).
\end{align}
Let us first consider the term $I'_1(y,r)$. By the weighted weak $(1,1)$ boundedness of vector-valued intrinsic square function (see Theorem \ref{weak}), we get
\begin{align}\label{I1prime}
I'_1(y,r)&\leq 2\cdot w(B(y,r))^{1/{\alpha}-1-1/q}\bigg\|\bigg(\sum_{j=1}^\infty\big|\mathcal S_{\gamma}(f^0_j)\big|^2\bigg)^{1/2}\bigg\|_{WL^1_w}\notag\\
&\leq C\cdot w(B(y,r))^{1/{\alpha}-1-1/q}
\bigg(\int_{B(y,2r)}\bigg(\sum_{j=1}^\infty\big|f_j(x)\big|^2\bigg)^{1/2}w(x)\,dx\bigg)\notag\\
&=C\cdot w(B(y,2r))^{1/{\alpha}-1-1/q}\bigg\|\bigg(\sum_{j=1}^\infty\big|f_j\big|^2\bigg)^{1/2}\cdot\chi_{B(y,2r)}\bigg\|_{L^1_w}\notag\\
&\times \frac{w(B(y,r))^{1/{\alpha}-1-1/q}}{w(B(y,2r))^{1/{\alpha}-1-1/q}}.
\end{align}
Moreover, since $1/{\alpha}-1-1/q<0$ and $w\in A_1$, then we apply doubling inequality (\ref{weights}) to obtain that
\begin{equation}\label{doubling2}
\frac{w(B(y,r))^{1/{\alpha}-1-1/q}}{w(B(y,2r))^{1/{\alpha}-1-1/q}}\leq C.
\end{equation}
Substituting the above inequality \eqref{doubling2} into \eqref{I1prime}, we thus obtain
\begin{equation}\label{WI1yr}
I'_1(y,r)\leq C\cdot w(B(y,2r))^{1/{\alpha}-1-1/q}\bigg\|\bigg(\sum_{j=1}^\infty\big|f_j\big|^2\bigg)^{1/2}\cdot\chi_{B(y,2r)}\bigg\|_{L^1_w}.
\end{equation}
As for the second term $I'_2(y,r)$, it follows directly from Chebyshev's inequality and the pointwise estimate \eqref{key estimate1} that
\begin{equation*}
\begin{split}
I'_2(y,r)&\leq2\cdot w(B(y,r))^{1/{\alpha}-1-1/q}\int_{B(y,r)}\bigg(\sum_{j=1}^\infty\big|\mathcal S_\gamma(f^\infty_j)(x)\big|^2\bigg)^{1/2}w(x)\,dx\\
&\leq C\cdot w(B(y,r))^{1/{\alpha}-1/q}
\sum_{l=1}^\infty\frac{1}{|B(y,2^{l+1}r)|}\int_{B(y,2^{l+1}r)}\bigg(\sum_{j=1}^\infty\big|f_j(z)\big|^2\bigg)^{1/2}dz.
\end{split}
\end{equation*}
Another application of $A_1$ condition on $w$ leads to that
\begin{equation*}
\begin{split}
&\frac{1}{|B(y,2^{l+1}r)|}\int_{B(y,2^{l+1}r)}\bigg(\sum_{j=1}^\infty\big|f_j(z)\big|^2\bigg)^{1/2}dz\\
&\leq C\frac{1}{w(B(y,2^{l+1}r))}\cdot\underset{z\in B(y,2^{l+1}r)}{\mbox{ess\,inf}}\;w(z)
\int_{B(y,2^{l+1}r)}\bigg(\sum_{j=1}^\infty\big|f_j(z)\big|^2\bigg)^{1/2}dz\\
&\leq C\frac{1}{w(B(y,2^{l+1}r))}\bigg(\int_{B(y,2^{l+1}r)}\bigg(\sum_{j=1}^\infty\big|f_j(z)\big|^2\bigg)^{1/2}w(z)\,dz\bigg).
\end{split}
\end{equation*}
Consequently,
\begin{align}\label{WI2yr}
I'_2(y,r)&\leq C\cdot w(B(y,r))^{1/{\alpha}-1/q}\notag\\
&\times\sum_{l=1}^\infty\bigg(\int_{B(y,2^{l+1}r)}\bigg(\sum_{j=1}^\infty\big|f_j(z)\big|^2\bigg)^{1/2}w(z)\,dz\bigg)\cdot w\big(B(y,2^{l+1}r)\big)^{-1}\notag\\
&=C\sum_{l=1}^\infty(B(y,2^{l+1}r))^{1/{\alpha}-1-1/q}\bigg\|\bigg(\sum_{j=1}^\infty\big|f_j\big|^2\bigg)^{1/2}\cdot\chi_{B(y,2^{l+1}r)}\bigg\|_{L^1_w}\notag\\
&\times\frac{w(B(y,r))^{1/{\alpha}-1/q}}{w(B(y,2^{l+1}r))^{1/{\alpha}-1/q}}.
\end{align}
Therefore by taking the $L^q_{\mu}$-norm of both sides of \eqref{Iprime}(with respect to the variable $y$), and then using Minkowski's inequality, \eqref{WI1yr} and \eqref{WI2yr}, we compute
\begin{equation*}
\begin{split}
&\bigg\|w(B(y,r))^{1/{\alpha}-1-1/q}\bigg\|\bigg(\sum_{j=1}^\infty\big|\mathcal S_{\gamma}(f_j)\big|^2\bigg)^{1/2}\cdot\chi_{B(y,r)}\bigg\|_{WL^1_w}\bigg\|_{L^q_{\mu}}\\
&\leq\big\|I'_1(y,r)\big\|_{L^q_{\mu}}+\big\|I'_2(y,r)\big\|_{L^q_{\mu}}\\
&\leq C\bigg\|w(B(y,2r))^{1/{\alpha}-1-1/q}\bigg\|\bigg(\sum_{j=1}^\infty\big|f_j\big|^2\bigg)^{1/2}\cdot\chi_{B(y,2r)}\bigg\|_{L^1_w}\bigg\|_{L^q_{\mu}}\\
&+C\sum_{l=1}^\infty\bigg\|w(B(y,2^{l+1}r))^{1/{\alpha}-1-1/q}
\bigg\|\bigg(\sum_{j=1}^\infty\big|f_j\big|^2\bigg)^{1/2}\cdot\chi_{B(y,2^{l+1}r)}\bigg\|_{L^1_w}\bigg\|_{L^q_{\mu}}
\times\frac{w(B(y,r))^{1/{\alpha}-1/q}}{w(B(y,2^{l+1}r))^{1/{\alpha}-1/q}}\\
&\leq C\bigg\|\bigg(\sum_{j=1}^\infty\big|f_j\big|^2\bigg)^{1/2}\bigg\|_{(L^1,L^q)^{\alpha}(w;\mu)}
+C\bigg\|\bigg(\sum_{j=1}^\infty\big|f_j\big|^2\bigg)^{1/2}\bigg\|_{(L^1,L^q)^{\alpha}(w;\mu)}
\times\sum_{l=1}^\infty\frac{w(B(y,r))^{1/{\alpha}-1/q}}{w(B(y,2^{l+1}r))^{1/{\alpha}-1/q}}\\
&\leq C\bigg\|\bigg(\sum_{j=1}^\infty\big|f_j\big|^2\bigg)^{1/2}\bigg\|_{(L^1,L^q)^{\alpha}(w;\mu)},
\end{split}
\end{equation*}
where in the last inequality we have used the estimate \eqref{psi1}. We end the proof by taking the supremum over all $r>0$.
\end{proof}

\section{Proof of Theorem \ref{mainthm:3}}
Given a real-valued function $b\in BMO(\mathbb R^n)$, we will follow the idea developed in \cite{alvarez,ding} and denote $F(\xi)=e^{\xi[b(x)-b(z)]}$, $\xi\in\mathbb C$. By the analyticity of $F(\xi)$ on $\mathbb C$ and the Cauchy integral formula, we first compute
\begin{equation*}
\begin{split}
b(x)-b(z)&=F'(0)=\frac{1}{2\pi i}\int_{|\xi|=1}\frac{F(\xi)}{\xi^2}\,d\xi\\
&=\frac{1}{2\pi}\int_0^{2\pi}e^{e^{i\theta}[b(x)-b(z)]}\cdot e^{-i\theta}d\theta.
\end{split}
\end{equation*}
Thus, for any $\varphi\in{\mathcal C}_\gamma$, $0<\gamma\le1$ and $j\in\mathbb Z^+$, we obtain
\begin{equation*}
\begin{split}
&\bigg|\int_{\mathbb R^n}\big[b(x)-b(z)\big]\varphi_t(y-z)f_j(z)\,dz\bigg|\\
=&\bigg|\frac{1}{2\pi}\int_0^{2\pi}\bigg(\int_{\mathbb R^n}\varphi_t(y-z)e^{-e^{i\theta}b(z)}f_j(z)\,dz\bigg)
e^{e^{i\theta}b(x)}\cdot e^{-i\theta}d\theta\bigg|\\
\leq&\frac{1}{2\pi}\int_0^{2\pi}\sup_{\varphi\in{\mathcal C}_\gamma}\bigg|\int_{\mathbb R^n}\varphi_t(y-z)e^{-e^{i\theta}b(z)}f_j(z)\,dz\bigg|
e^{\cos\theta\cdot b(x)}d\theta\\
\leq&\frac{1}{2\pi}\int_0^{2\pi}\mathcal A_\gamma\big(e^{-e^{i\theta}b}\cdot f_j\big)(y,t)\cdot e^{\cos\theta\cdot b(x)}d\theta.
\end{split}
\end{equation*}
From this, it follows that
\begin{equation*}
\big|\big[b,\mathcal S_\gamma\big](f_j)(x)\big|\leq\frac{1}{2\pi}\int_0^{2\pi}
\mathcal S_\gamma\big(e^{-e^{i\theta}b}\cdot f_j\big)(x)\cdot e^{\cos\theta\cdot b(x)}d\theta.
\end{equation*}
Moreover, by using standard duality argument and Cauchy--Schwarz inequality, we compute
\begin{equation*}
\begin{split}
&\bigg(\sum_{j=1}^\infty\big|\big[b,\mathcal S_\gamma\big](f_j)(x)\big|^2\bigg)^{1/2}\\
&\leq\frac{1}{2\pi}\Bigg(\sum_{j=1}^\infty\left|\int_0^{2\pi}
\mathcal S_\gamma\big(e^{-e^{i\theta}b}\cdot f_j\big)(x)\cdot e^{\cos\theta\cdot b(x)}d\theta\right|^2\Bigg)^{1/2}\\
&=\frac{1}{2\pi}\sup_{(\sum_{j=1}^\infty|\zeta_j|^2)^{1/2}\leq1}\sum_{j=1}^\infty
\left(\int_0^{2\pi}\mathcal S_\gamma\big(e^{-e^{i\theta}b}\cdot f_j\big)(x)\cdot e^{\cos\theta\cdot b(x)}d\theta\cdot\zeta_j\right)\\
&\leq\frac{1}{2\pi}\int_0^{2\pi}\sup_{(\sum_{j=1}^\infty|\zeta_j|^2)^{1/2}\leq1}
\Bigg(\sum_{j=1}^\infty\mathcal S_\gamma\big(e^{-e^{i\theta}b}\cdot f_j\big)(x)\cdot e^{\cos\theta\cdot b(x)}\cdot\zeta_j\Bigg)d\theta\\
&\leq\frac{1}{2\pi}\int_0^{2\pi}\bigg(\sum_{j=1}^\infty\Big|\mathcal S_\gamma\big(e^{-e^{i\theta}b}\cdot f_j\big)(x)\Big|^2\bigg)^{1/2}
\cdot e^{\cos\theta\cdot b(x)}d\theta.
\end{split}
\end{equation*}
Therefore, applying the $L^p_w$-boundedness of vector-valued intrinsic square function (see Theorem \ref{strong}),
and the same method as proving Theorem 1 in \cite{ding}, we can also show the following result.
\begin{theorem}\label{commutator thm}
Let $0<\gamma\le1$, $1<p<\infty$ and $w\in A_p$. Then there exists a constant $C>0$ independent of $\vec{f}=(f_1,f_2,\ldots)$ such that
\begin{equation*}
\bigg\|\bigg(\sum_{j=1}^\infty\big|\big[b,\mathcal S_\gamma\big](f_j)\big|^2\bigg)^{1/2}\bigg\|_{L^p_w}
\leq C \bigg\|\bigg(\sum_{j=1}^\infty\big|f_j\big|^2\bigg)^{1/2}\bigg\|_{L^p_w}
\end{equation*}
provided that $b\in BMO(\mathbb R^n)$.
\end{theorem}

To prove our main theorem in this section, we also need the following lemma about $BMO(\mathbb R^n)$ functions.
\begin{lemma}\label{BMO}
Let $b$ be a function in $BMO(\mathbb R^n)$. Then

$(i)$ For every ball $B$ in $\mathbb R^n$ and for all $l\in\mathbb Z^+$,
\begin{equation*}
\big|b_{2^{l+1}B}-b_B\big|\leq C\cdot(l+1)\|b\|_*.
\end{equation*}

$(ii)$ For every ball $B$ in $\mathbb R^n$ and for all $w\in A_p$ with $1\leq p<\infty$,
\begin{equation*}
\bigg(\int_B\big|b(x)-b_B\big|^pw(x)\,dx\bigg)^{1/p}\leq C\|b\|_*\cdot w(B)^{1/p}.
\end{equation*}
\end{lemma}
\begin{proof}
For the proof of $(i)$, we refer the reader to \cite{stein2}. For the proof of $(ii)$, we refer the reader to \cite{wang}.
\end{proof}

\begin{proof}[Proof of Theorem $\ref{mainthm:3}$]
Let $1<p\leq\alpha<q\leq\infty$ and $\big(\sum_{j=1}^\infty|f_j|^2\big)^{1/2}\in(L^p,L^q)^{\alpha}(w;\mu)$ with $w\in A_p$ and $\mu\in\Delta_2$.
For each fixed ball $B=B(y,r)\subset\mathbb R^n$, as before, we represent $f_j$ as $f_j=f^0_j+f^\infty_j$, where $f^0_j=f_j\cdot\chi_{2B}$ and $2B=B(y,2r)\subset\mathbb R^n$, $j=1,2,\ldots$. Then we write
\begin{align}\label{J}
&w(B(y,r))^{1/{\alpha}-1/p-1/q}\bigg\|\bigg(\sum_{j=1}^\infty\big|\big[b,\mathcal S_\gamma\big](f_j)\big|^2\bigg)^{1/2}\cdot\chi_{B(y,r)}\bigg\|_{L^p_w}\notag\\
&=w(B(y,r))^{1/{\alpha}-1/p-1/q}\bigg(\int_{B(y,r)}\bigg(\sum_{j=1}^\infty
\big|\big[b,\mathcal S_\gamma\big](f_j)(x)\big|^2\bigg)^{p/2}w(x)\,dx\bigg)^{1/p}\notag\\
&\leq w(B(y,r))^{1/{\alpha}-1/p-1/q}\bigg(\int_{B(y,r)}\bigg(\sum_{j=1}^\infty\big|\big[b,\mathcal S_\gamma\big](f^0_j)(x)\big|^2\bigg)^{p/2}w(x)\,dx\bigg)^{1/p}\notag\\
&+w(B(y,r))^{1/{\alpha}-1/p-1/q}\bigg(\int_{B(y,r)}\bigg(\sum_{j=1}^\infty\big|\big[b,\mathcal S_\gamma\big](f^\infty_j)(x)\big|^2\bigg)^{p/2}w(x)\,dx\bigg)^{1/p}\notag\\
&:=J_1(y,r)+J_2(y,r).
\end{align}
By using Theorem \ref{commutator thm} and the inequality (\ref{doubling1}), we obtain
\begin{align}\label{J1yr}
J_1(y,r)&\leq w(B(y,r))^{1/{\alpha}-1/p-1/q}\bigg\|\bigg(\sum_{j=1}^\infty\big|\big[b,\mathcal S_\gamma\big](f^0_j)\big|^2\bigg)^{1/2}\bigg\|_{L^p_w}\notag\\
&\leq C\cdot w(B(y,r))^{1/{\alpha}-1/p-1/q}
\bigg(\int_{B(y,2r)}\bigg(\sum_{j=1}^\infty\big|f_j(x)\big|^2\bigg)^{p/2}w(x)\,dx\bigg)^{1/p}\notag\\
&=C\cdot w(B(y,2r))^{1/{\alpha}-1/p-1/q}\bigg\|\bigg(\sum_{j=1}^\infty\big|f_j\big|^2\bigg)^{1/2}\cdot\chi_{B(y,2r)}\bigg\|_{L^p_w}\times \frac{w(B(y,r))^{1/{\alpha}-1/p-1/q}}{w(B(y,2r))^{1/{\alpha}-1/p-1/q}}\notag\\
&\leq C\cdot w(B(y,2r))^{1/{\alpha}-1/p-1/q}\bigg\|\bigg(\sum_{j=1}^\infty\big|f_j\big|^2\bigg)^{1/2}\cdot\chi_{B(y,2r)}\bigg\|_{L^p_w}.
\end{align}
Let us now turn to the estimate of $J_2(y,r)$. For any given $\varphi\in{\mathcal C}_\gamma$, $0<\gamma\le1$, $j=1,2,\ldots$, and $(\xi,t)\in\Gamma(x)$ with $x\in B(y,r)$, we have
\begin{equation*}
\begin{split}
\sup_{\varphi\in{\mathcal C}_\gamma}\bigg|\int_{\mathbb R^n}\big[b(x)-b(z)\big]\varphi_t(\xi-z)f^\infty_j(z)\,dz\bigg|&\le
\big|b(x)-b_{B(y,r)}\big|\cdot\sup_{\varphi\in{\mathcal C}_\gamma}\bigg|\int_{\mathbb R^n}\varphi_t(\xi-z)f^\infty_j(z)\,dz\bigg|\\
&+\sup_{\varphi\in{\mathcal C}_\gamma}\bigg|\int_{\mathbb R^n}\big[b_{B(y,r)}-b(z)\big]\varphi_t(\xi-z)f^\infty_j(z)\,dz\bigg|.
\end{split}
\end{equation*}
By definition, we can see that
\begin{equation*}
\begin{split}
\big|\big[b,\mathcal S_\gamma\big](f^\infty_j)(x)\big|&\leq\big|b(x)-b_{B(y,r)}\big|\cdot\mathcal S_\gamma(f^\infty_j)(x)
+\mathcal S_\gamma\Big([b_{B(y,r)}-b]f^\infty_j\Big)(x).
\end{split}
\end{equation*}
From this and Minkowski' inequality for series, we further obtain for any $x\in B(y,r)$,
\begin{equation*}
\begin{split}
\bigg(\sum_{j=1}^\infty\big|\big[b,\mathcal S_\gamma\big](f^\infty_j)(x)\big|^2\bigg)^{1/2}
&\leq\big|b(x)-b_{B(y,r)}\big|\bigg(\sum_{j=1}^\infty\big|\mathcal S_\gamma(f^{\infty}_j)(x)\big|^2\bigg)^{1/2}\\
&+\bigg(\sum_{j=1}^\infty\Big|\mathcal S_\gamma\Big([b_{B(y,r)}-b]f^\infty_j\Big)(x)\Big|^2\bigg)^{1/2}.
\end{split}
\end{equation*}
Fix $x\in B(y,r)$, the following estimate is known from \eqref{key estimate1}:
\begin{equation}\label{Keyw}
\bigg(\sum_{j=1}^\infty\big|\mathcal S_\gamma(f^\infty_j)(x)\big|^2\bigg)^{1/2}
\leq C\sum_{l=1}^\infty\frac{1}{|B(y,2^{l+1}r)|}\int_{B(y,2^{l+1}r)}\bigg(\sum_{j=1}^\infty\big|f_j(z)\big|^2\bigg)^{1/2}dz.
\end{equation}
Using the same procedure as in the proof of \eqref{key estimate1}, for any $\varphi\in{\mathcal C}_\gamma$, $0<\gamma\le1$, $j=1,2,\ldots$, and $(\xi,t)\in\Gamma(x)$ with $x\in B(y,r)$, we can also show that
\begin{align}\label{Key3}
&\bigg|\int_{\mathbb R^n}\big[b_{B(y,r)}-b(z)\big]\varphi_t(\xi-z)f^\infty_j(z)\,dz\bigg|\notag\\
&=\bigg|\int_{B(y,2r)^c}\big[b_{B(y,r)}-b(z)\big]\varphi_t(\xi-z)f_j(z)\,dz\bigg|\notag\\
&\leq C\cdot t^{-n}\int_{B(y,2r)^c\cap\{z:|\xi-z|\le t\}}\big|b(z)-b_{B(y,r)}\big|\big|f_j(z)\big|\,dz\notag\\
&\leq C\cdot t^{-n}\sum_{l=1}^\infty\int_{B(y,2^{l+1}r)\backslash B(y,2^{l}r)\cap\{z:|\xi-z|\le t\}}\big|b(z)-b_{B(y,r)}\big|\big|f_j(z)\big|\,dz.
\end{align}
Hence, for any $x\in B(y,r)$, by using the inequalities \eqref{Key3} and \eqref{Key2} together with Minkowski's inequality for integrals, we can deduce
\begin{equation*}
\begin{split}
&\mathcal S_\gamma\Big([b_{B(y,r)}-b]f^\infty_j\Big)(x)
=\left(\iint_{\Gamma(x)}\sup_{\varphi\in{\mathcal C}_\gamma}
\bigg|\int_{\mathbb R^n}\big[b_{B(y,r)}-b(z)\big]\varphi_t(\xi-z)f^\infty_j(z)\,dz\bigg|^2\frac{d\xi dt}{t^{n+1}}\right)^{1/2}\\
&\leq C\left(\int_{2^{l-2}r}^\infty\int_{|x-\xi|<t}\bigg|t^{-n}\sum_{l=1}^\infty\int_{B(y,2^{l+1}r)\backslash B(y,2^{l}r)}\big|b(z)-b_{B(y,r)}\big|\big|f_j(z)\big|\,dz\bigg|^2\frac{d\xi dt}{t^{n+1}}\right)^{1/2}\\
&\le C\left(\sum_{l=1}^\infty\int_{B(y,2^{l+1}r)\backslash B(y,2^{l}r)}\big|b(z)-b_{B(y,r)}\big|\big|f_j(z)\big|\,dz\right)
\left(\int_{2^{l-2}r}^\infty\frac{dt}{t^{2n+1}}\right)^{1/2}\\
&\leq C\sum_{l=1}^\infty\frac{1}{|B(y,2^{l+1}r)|}\int_{B(y,2^{l+1}r)\backslash B(y,2^{l}r)}\big|b(z)-b_{B(y,r)}\big|\big|f_j(z)\big|\,dz.
\end{split}
\end{equation*}
Therefore, by duality and Cauchy--Schwarz inequality, we get
\begin{align}\label{key estimate2}
&\bigg(\sum_{j=1}^\infty\Big|\mathcal S_\gamma\Big([b_{B(y,r)}-b]f^\infty_j\Big)(x)\Big|^2\bigg)^{1/2}\notag\\
&\leq C\Bigg(\sum_{j=1}^\infty\bigg|\sum_{l=1}^\infty\frac{1}{|B(y,2^{l+1}r)|}
\int_{B(y,2^{l+1}r)\backslash B(y,2^{l}r)}\big|b(z)-b_{B(y,r)}\big|\big|f_j(z)\big|\,dz\bigg|^2\Bigg)^{1/2}\notag\\
&=C\sup_{(\sum_{j=1}^\infty|\zeta_j|^2)^{1/2}\leq1}\sum_{j=1}^\infty\bigg(\sum_{l=1}^\infty\frac{1}{|B(y,2^{l+1}r)|}
\int_{B(y,2^{l+1}r)}\big|b(z)-b_{B(y,r)}\big|\big|f_j(z)\big|\,dz\cdot\zeta_j\bigg)\notag\\
&\leq C\sum_{l=1}^\infty\frac{1}{|B(y,2^{l+1}r)|}\int_{B(y,2^{l+1}r)}\sup_{(\sum_{j=1}^\infty|\zeta_j|^2)^{1/2}\leq1}
\bigg(\sum_{j=1}^\infty\big|b(z)-b_{B(y,r)}\big|\big|f_j(z)\big|\cdot\zeta_j\bigg)dz\notag\\
&\leq C\sum_{l=1}^\infty\frac{1}{|B(y,2^{l+1}r)|}\int_{B(y,2^{l+1}r)}\big|b(z)-b_{B(y,r)}\big|\bigg(\sum_{j=1}^\infty\big|f_j(z)\big|^2\bigg)^{1/2}dz.
\end{align}
Consequently, from the above two pointwise estimates \eqref{Keyw} and \eqref{key estimate2}, it follows that
\begin{equation*}
\begin{split}
J_2(y,r)&\leq C\cdot w(B(y,r))^{1/{\alpha}-1/p-1/q}\bigg(\int_{B(y,r)}\big|b(x)-b_{B(y,r)}\big|^pw(x)\,dx\bigg)^{1/p}\\
&\times\bigg(\sum_{l=1}^\infty\frac{1}{|B(y,2^{l+1}r)|}\int_{B(y,2^{l+1}r)}\bigg(\sum_{j=1}^\infty\big|f_j(z)\big|^2\bigg)^{1/2}dz\bigg)\\
&+C\cdot w(B(y,r))^{1/{\alpha}-1/q}\sum_{l=1}^\infty\frac{1}{|B(y,2^{l+1}r)|}\int_{B(y,2^{l+1}r)}\big|b_{B(y,2^{l+1}r)}-b_{B(y,r)}\big|
\cdot\bigg(\sum_{j=1}^\infty\big|f_j(z)\big|^2\bigg)^{1/2}dz\\
&+C\cdot w(B(y,r))^{1/{\alpha}-1/q}
\sum_{l=1}^\infty\frac{1}{|B(y,2^{l+1}r)|}\int_{B(y,2^{l+1}r)}\big|b(z)-b_{B(y,2^{l+1}r)}\big|
\cdot\bigg(\sum_{j=1}^\infty\big|f_j(z)\big|^2\bigg)^{1/2}dz\\
&:=J_3(y,r)+J_4(y,r)+J_5(y,r).
\end{split}
\end{equation*}
Below we will give the estimates of $J_3(y,r)$, $J_4(y,r)$ and $J_5(y,r)$, respectively. Using $(ii)$ of Lemma \ref{BMO}, H\"older's inequality and the $A_p$ condition on $w$, we obtain
\begin{equation*}
\begin{split}
J_3(y,r)&\leq C\|b\|_*\cdot w(B(y,r))^{1/{\alpha}-1/q}
\times\sum_{l=1}^\infty\bigg(\frac{1}{|B(y,2^{l+1}r)|}\int_{B(y,2^{l+1}r)}\bigg(\sum_{j=1}^\infty\big|f_j(z)\big|^2\bigg)^{1/2}dz\bigg)\\
&\leq C\|b\|_*\cdot w(B(y,r))^{1/{\alpha}-1/q}\sum_{l=1}^\infty\frac{1}{|B(y,2^{l+1}r)|}
\bigg(\int_{B(y,2^{l+1}r)}\bigg(\sum_{j=1}^\infty\big|f_j(z)\big|^2\bigg)^{p/2}w(z)\,dz\bigg)^{1/p}\\
&\times\bigg(\int_{B(y,2^{l+1}r)}w(z)^{-{p'}/p}\,dz\bigg)^{1/{p'}}\\
&\leq C\|b\|_*\cdot w(B(y,r))^{1/{\alpha}-1/q}\\
&\times\sum_{l=1}^\infty\bigg(\int_{B(y,2^{l+1}r)}\bigg(\sum_{j=1}^\infty\big|f_j(z)\big|^2\bigg)^{p/2}w(z)\,dz\bigg)^{1/p}
\cdot w\big(B(y,2^{l+1}r)\big)^{-1/p}.
\end{split}
\end{equation*}
On the other hand, applying $(i)$ of Lemma \ref{BMO}, H\"older's inequality and the $A_p$ condition on $w$, we can deduce that
\begin{equation*}
\begin{split}
J_4(y,r)&\leq C\|b\|_*\cdot w(B(y,r))^{1/{\alpha}-1/q}
\times\sum_{l=1}^\infty\frac{(l+1)}{|B(y,2^{l+1}r)|}
\int_{B(y,2^{l+1}r)}\bigg(\sum_{j=1}^\infty\big|f_j(z)\big|^2\bigg)^{1/2}dz\\
&\leq C\|b\|_*\cdot w(B(y,r))^{1/{\alpha}-1/q}
\sum_{l=1}^\infty\frac{(l+1)}{|B(y,2^{l+1}r)|}
\bigg(\int_{B(y,2^{l+1}r)}\bigg(\sum_{j=1}^\infty\big|f_j(z)\big|^2\bigg)^{p/2}w(z)\,dz\bigg)^{1/p}\\
&\times\bigg(\int_{B(y,2^{l+1}r)}w(z)^{-{p'}/p}\,dz\bigg)^{1/{p'}}\\
&\leq C\|b\|_*\cdot w(B(y,r))^{1/{\alpha}-1/q}\\
&\times\sum_{l=1}^\infty\big(l+1\big)\cdot\bigg(\int_{B(y,2^{l+1}r)}\bigg(\sum_{j=1}^\infty\big|f_j(z)\big|^2\bigg)^{p/2}w(z)\,dz\bigg)^{1/p}
\cdot w\big(B(y,2^{l+1}r)\big)^{-1/p}.
\end{split}
\end{equation*}
It remains to estimate the last term $J_5(y,r)$. An application of H\"older's inequality gives us that
\begin{equation*}
\begin{split}
J_5(y,r)&\leq C\cdot w(B(y,r))^{1/{\alpha}-1/q}\sum_{l=1}^\infty\frac{1}{|B(y,2^{l+1}r)|}
\bigg(\int_{B(y,2^{l+1}r)}\bigg(\sum_{j=1}^\infty\big|f_j(z)\big|^2\bigg)^{p/2}w(z)\,dz\bigg)^{1/p}\\
&\times\left(\int_{B(y,2^{l+1}r)}\big|b(z)-b_{B(y,2^{l+1}r)}\big|^{p'}w(z)^{-{p'}/p}\,dz\right)^{1/{p'}}.
\end{split}
\end{equation*}
If we denote $\nu(z)=w(z)^{-{p'}/p}$, then we know that the weight $\nu(z)$ belongs to $A_{p'}$ whenever $w\in A_p$(see \cite{duoand,garcia}). From this fact together with $(ii)$ of Lemma \ref{BMO} and the $A_p$ condition, it follows that
\begin{align}\label{WH}
\bigg(\int_{B(y,2^{l+1}r)}\big|b(z)-b_{B(y,2^{l+1}r)}\big|^{p'}\nu(z)\,dz\bigg)^{1/{p'}}
&\leq C\|b\|_*\cdot\nu\big(B(y,2^{l+1}r)\big)^{1/{p'}}\notag\\
&=C\|b\|_*\cdot\bigg(\int_{B(y,2^{l+1}r)}w(z)^{-{p'}/p}\,dz\bigg)^{1/{p'}}\notag\\
&\leq C\|b\|_*\cdot\frac{|B(y,2^{l+1}r)|}{w(B(y,2^{l+1}r))^{1/p}}.
\end{align}
Therefore, in view of \eqref{WH}, we can see that
\begin{equation*}
\begin{split}
J_5(y,r)&\leq C\|b\|_*\cdot w(B(y,r))^{1/{\alpha}-1/q}\\
&\times\sum_{l=1}^\infty\bigg(\int_{B(y,2^{l+1}r)}\bigg(\sum_{j=1}^\infty\big|f_j(z)\big|^2\bigg)^{p/2}w(z)\,dz\bigg)^{1/p}
\cdot w\big(B(y,2^{l+1}r)\big)^{-1/p}.
\end{split}
\end{equation*}
Summarizing the above discussions, we conclude that
\begin{align}\label{J2yr}
J_2(y,r)&\leq C\|b\|_*\cdot w(B(y,r))^{1/{\alpha}-1/q}\notag\\
&\times\sum_{l=1}^\infty\big(l+1\big)\cdot\bigg(\int_{B(y,2^{l+1}r)}\bigg(\sum_{j=1}^\infty\big|f_j(z)\big|^2\bigg)^{p/2}w(z)\,dz\bigg)^{1/p}
\cdot w\big(B(y,2^{l+1}r)\big)^{-1/p}\notag\\
&=C\sum_{l=1}^\infty w(B(y,2^{l+1}r))^{1/{\alpha}-1/p-1/q}\bigg\|\bigg(\sum_{j=1}^\infty\big|f_j\big|^2\bigg)^{1/2}\cdot\chi_{B(y,2^{l+1}r)}\bigg\|_{L^p_w}\notag\\
&\times\big(l+1\big)\cdot\frac{w(B(y,r))^{1/{\alpha}-1/q}}{w(B(y,2^{l+1}r))^{1/{\alpha}-1/q}}.
\end{align}
Notice that when $w\in A_p$ with $1\leq p<\infty$, one has $w\in A_\infty$. Thus, by using inequality (\ref{compare}) with exponent $\delta^*>0$ together with our assumption that $\alpha<q$, we obtain
\begin{align}\label{psi3}
\sum_{l=1}^\infty\big(l+1\big)\cdot\frac{w(B(y,r))^{1/{\alpha}-1/q}}{w(B(y,2^{l+1}r))^{1/{\alpha}-1/q}}
&\leq C\sum_{l=1}^\infty\big(l+1\big)\cdot\left(\frac{|B(y,r)|}{|B(y,2^{l+1}r)|}\right)^{\delta^*(1/{\alpha}-1/q)}\notag\\
&= C\sum_{l=1}^\infty\big(l+1\big)\cdot\left(\frac{1}{2^{(l+1)n}}\right)^{\delta^*(1/{\alpha}-1/q)}\notag\\
&\leq C,
\end{align}
where the last series is convergent since the exponent $\delta^*(1/{\alpha}-1/q)$ is positive. Therefore by taking the $L^q_{\mu}$-norm of both sides of \eqref{J}(with respect to the variable $y$), and then using Minkowski's inequality, \eqref{J1yr}, \eqref{J2yr} and \eqref{psi3}, we can get
\begin{equation*}
\begin{split}
&\bigg\|w(B(y,r))^{1/{\alpha}-1/p-1/q}\bigg\|\bigg(\sum_{j=1}^\infty\big|\big[b,\mathcal S_\gamma\big](f_j)\big|^2\bigg)^{1/2}\cdot\chi_{B(y,r)}\bigg\|_{L^p_w}\bigg\|_{L^q_{\mu}}\\
&\leq\big\|J_1(y,r)\big\|_{L^q_{\mu}}+\big\|J_2(y,r)\big\|_{L^q_{\mu}}\\
&\leq C\bigg\|w(B(y,2r))^{1/{\alpha}-1/p-1/q}\bigg\|\bigg(\sum_{j=1}^\infty\big|f_j\big|^2\bigg)^{1/2}\cdot\chi_{B(y,2r)}\bigg\|_{L^p_w}\bigg\|_{L^q_{\mu}}\\
&+C\sum_{l=1}^\infty\bigg\|w(B(y,2^{l+1}r))^{1/{\alpha}-1/p-1/q}
\bigg\|\bigg(\sum_{j=1}^\infty\big|f_j\big|^2\bigg)^{1/2}\cdot\chi_{B(y,2^{l+1}r)}\bigg\|_{L^p_w}\bigg\|_{L^q_{\mu}}\\
&\times\big(l+1\big)\cdot\frac{w(B(y,r))^{1/{\alpha}-1/q}}{w(B(y,2^{l+1}r))^{1/{\alpha}-1/q}}\\
&\leq C\bigg\|\bigg(\sum_{j=1}^\infty\big|f_j\big|^2\bigg)^{1/2}\bigg\|_{(L^p,L^q)^{\alpha}(w;\mu)}
+C\bigg\|\bigg(\sum_{j=1}^\infty\big|f_j\big|^2\bigg)^{1/2}\bigg\|_{(L^p,L^q)^{\alpha}(w;\mu)}\\
&\times\sum_{l=1}^\infty\big(l+1\big)\cdot\frac{w(B(y,r))^{1/{\alpha}-1/q}}{w(B(y,2^{l+1}r))^{1/{\alpha}-1/q}}\\
&\leq C\bigg\|\bigg(\sum_{j=1}^\infty\big|f_j\big|^2\bigg)^{1/2}\bigg\|_{(L^p,L^q)^{\alpha}(w;\mu)}.
\end{split}
\end{equation*}
Thus, by taking the supremum over all $r>0$, we complete the proof of Theorem \ref{mainthm:3}.
\end{proof}

\section{Proof of Theorem \ref{mainthm:4}}
To show Theorem \ref{mainthm:4}, we first establish the following endpoint estimate for vector-valued commutator (\ref{vectorvaluedc}) in the weighted Lebesgue space $L^1_w(\mathbb R^n)$.

\begin{theorem}\label{mainthm:5}
Let $0<\gamma\leq1$, $p=1$, $w\in A_1$ and $b\in BMO(\mathbb R^n)$. Then for any given $\sigma>0$, there exists a constant $C>0$ independent of $\vec{f}=(f_1,f_2,\ldots)$ and $\sigma>0$ such that
\begin{equation}\label{end}
w\bigg(\bigg\{x\in\mathbb R^n:\bigg(\sum_{j=1}^\infty\big|\big[b,\mathcal S_{\gamma}\big](f_j)(x)\big|^2\bigg)^{1/2}>\sigma\bigg\}\bigg)
\leq C\int_{\mathbb R^n}\Phi\bigg(\frac{\|\vec{f}(x)\|_{\ell^2}}{\sigma}\bigg)\cdot w(x)\,dx,
\end{equation}
where $\Phi(t)=t\cdot(1+\log^+t)$ and $\big\|\vec{f}(x)\big\|_{\ell^2}=\Big(\sum_{j=1}^\infty|f_j(x)|^2\Big)^{1/2}$.
\end{theorem}

\begin{proof}
Inspired by the works in \cite{perez2,perez4,zhang}, for any fixed $\sigma>0$, we apply the Calder\'on--Zygmund decomposition of $\vec{f}=(f_1,f_2,\ldots)$ at height $\sigma$ to obtain a collection of disjoint non-overlapping dyadic cubes $\{Q_i\}$ such that the following property holds (see \cite{stein,perez4})
\begin{equation}\label{decomposition}
\sigma<\frac{1}{|Q_i|}\int_{Q_i}\bigg(\sum_{j=1}^\infty\big|f_j(y)\big|^2\bigg)^{1/2}dy\leq 2^n\cdot\sigma,
\end{equation}
where $Q_i=Q(c_i,\ell_i)$ denotes the cube centered at $c_i$ with side length $\ell_i$ and all cubes are
assumed to have their sides parallel to the coordinate axes. If we set $E=\bigcup_i Q_i$, then
\begin{equation*}
\bigg(\sum_{j=1}^\infty\big|f_j(y)\big|^2\bigg)^{1/2}\leq\sigma, \quad \mbox{a.e. }\, x\in\mathbb R^n\backslash E.
\end{equation*}
Now we proceed to construct vector-valued version of the Calder\'on--Zygmund decomposition. Define two vector-valued functions $\vec{g}=(g_1,g_2,\ldots)$ and $\vec{h}=(h_1,h_2,\ldots)$ as follows:
\begin{equation*}
g_j(x):=
\begin{cases}
f_j(x) &  \mbox{if}\;\; x\in E^c,\\
\displaystyle\frac{1}{|Q_i|}\int_{Q_i}f_j(y)\,dy    &  \mbox{if}\;\; x\in Q_i,
\end{cases}
\end{equation*}
and
\begin{equation*}
h_j(x):=f_j(x)-g_j(x)=\sum_i h_{ij}(x),\quad j=1,2,\dots,
\end{equation*}
where $h_{ij}(x)=h_j(x)\cdot\chi_{Q_i}(x)=\big(f_j(x)-g_j(x)\big)\cdot\chi_{Q_i}(x)$. Then one has
\begin{equation}\label{pointwise estimate g}
\bigg(\sum_{j=1}^\infty\big|g_j(x)\big|^2\bigg)^{1/2}\leq C\cdot\sigma, \quad \mbox{a.e. }\, x\in\mathbb R^n,
\end{equation}
and
\begin{equation}\label{f=g+h}
\vec{f}=\vec{g}+\vec{h}:=(g_1+h_1,g_2+h_2,\ldots).
\end{equation}
Obviously, $h_{ij}$ is supported on $Q_i$, $i,j=1,2,\dots$,
\begin{equation*}
\int_{\mathbb R^n}h_{ij}(x)\,dx=0,\quad \mbox{and}\quad \big\|h_{ij}\big\|_{L^1}=\int_{\mathbb R^n}\big|h_{ij}(x)\big|\,dx\leq 2\int_{Q_i}\big|f_j(x)\big|\,dx
\end{equation*}
according to the above decomposition. By \eqref{f=g+h} and Minkowski's inequality,
\begin{equation*}
\bigg(\sum_{j=1}^\infty\big|\big[b,\mathcal S_{\gamma}\big](f_j)(x)\big|^2\bigg)^{1/2}
\leq\bigg(\sum_{j=1}^\infty\big|\big[b,\mathcal S_{\gamma}\big](g_j)(x)\big|^2\bigg)^{1/2}+
\bigg(\sum_{j=1}^\infty\big|\big[b,\mathcal S_{\gamma}\big](h_j)(x)\big|^2\bigg)^{1/2}.
\end{equation*}
Then we can write
\begin{equation*}
\begin{split}
&w\bigg(\bigg\{x\in\mathbb R^n:\bigg(\sum_{j=1}^\infty\big|\big[b,\mathcal S_{\gamma}\big](f_j)(x)\big|^2\bigg)^{1/2}>\sigma\bigg\}\bigg)\\
&\leq w\bigg(\bigg\{x\in\mathbb R^n:\bigg(\sum_{j=1}^\infty\big|\big[b,\mathcal S_{\gamma}\big](g_j)(x)\big|^2\bigg)^{1/2}>\sigma/2\bigg\}\bigg)\\
&+w\bigg(\bigg\{x\in\mathbb R^n:\bigg(\sum_{j=1}^\infty\big|\big[b,\mathcal S_{\gamma}\big](h_j)(x)\big|^2\bigg)^{1/2}>\sigma/2\bigg\}\bigg)\\
&:=K_1+K_2.
\end{split}
\end{equation*}
Observe that $w\in A_1\subset A_2$. Applying Chebyshev's inequality and Theorem \ref{commutator thm}, we obtain
\begin{equation*}
K_1\leq \frac{4}{\sigma^2}\cdot\bigg\|\bigg(\sum_{j=1}^\infty\big|\big[b,\mathcal S_{\gamma}\big](g_j)\big|^2\bigg)^{1/2}\bigg\|^2_{L^2_w}
\leq \frac{C}{\sigma^2}\cdot\bigg\|\bigg(\sum_{j=1}^\infty\big|g_j\big|^2\bigg)^{1/2}\bigg\|^2_{L^2_w}.
\end{equation*}
Moreover, in view of (\ref{pointwise estimate g}), one has
\begin{align*}
&\bigg\|\bigg(\sum_{j=1}^\infty\big|g_j\big|^2\bigg)^{1/2}\bigg\|^2_{L^2_w}
\leq C\cdot\sigma\int_{\mathbb R^n}\bigg(\sum_{j=1}^\infty\big|g_j(x)\big|^2\bigg)^{1/2}w(x)\,dx\\
&\leq C\cdot\sigma\Bigg(\int_{E^c}\bigg(\sum_{j=1}^\infty\big|f_j(x)\big|^2\bigg)^{1/2}w(x)\,dx
+\int_{\bigcup_i Q_i}\bigg(\sum_{j=1}^\infty\big|g_j(x)\big|^2\bigg)^{1/2}w(x)\,dx\Bigg).
\end{align*}
Recall that $g_j(x)=\frac{1}{|Q_i|}\int_{Q_i}f_j(y)\,dy$ when $x\in Q_i$. As before, by using duality and Cauchy--Schwarz inequality, we can see the following estimate is valid for all $x\in Q_i$.
\begin{equation}\label{vectorg}
\bigg(\sum_{j=1}^\infty\big|g_j(x)\big|^2\bigg)^{1/2}\leq\frac{1}{|Q_i|}\int_{Q_i}\bigg(\sum_{j=1}^\infty\big|f_j(y)\big|^2\bigg)^{1/2}dy.
\end{equation}
This estimate \eqref{vectorg} along with the $A_1$ condition yields
\begin{align}\label{g}
&\bigg\|\bigg(\sum_{j=1}^\infty\big|g_j\big|^2\bigg)^{1/2}\bigg\|^2_{L^2_w}\le C\cdot\sigma\Bigg(\int_{\mathbb R^n}\bigg(\sum_{j=1}^\infty\big|f_j(x)\big|^2\bigg)^{1/2}w(x)\,dx
+\sum_i\frac{w(Q_i)}{|Q_i|}\int_{Q_i}\bigg(\sum_{j=1}^\infty\big|f_j(y)\big|^2\bigg)^{1/2}dy\Bigg)\notag\\
&\leq C\cdot\sigma\Bigg(\int_{\mathbb R^n}\bigg(\sum_{j=1}^\infty\big|f_j(x)\big|^2\bigg)^{1/2}w(x)\,dx
+\sum_i\underset{y\in Q_i}{\mbox{ess\,inf}}\,w(y)\int_{Q_i}\bigg(\sum_{j=1}^\infty\big|f_j(y)\big|^2\bigg)^{1/2}dy\Bigg)\notag\\
&\leq C\cdot\sigma\Bigg(\int_{\mathbb R^n}\bigg(\sum_{j=1}^\infty\big|f_j(x)\big|^2\bigg)^{1/2}w(x)\,dx
+\int_{\bigcup_i Q_i}\bigg(\sum_{j=1}^\infty\big|f_j(y)\big|^2\bigg)^{1/2}w(y)\,dy\Bigg)\notag\\
&\le C\cdot\sigma\int_{\mathbb R^n}\bigg(\sum_{j=1}^\infty\big|f_j(x)\big|^2\bigg)^{1/2}w(x)\,dx.
\end{align}
So we have
\begin{equation*}
K_1\leq C\int_{\mathbb R^n}\frac{\|\vec{f}(x)\|_{\ell^2}}{\sigma}\cdot w(x)\,dx\leq C\int_{\mathbb R^n}\Phi\bigg(\frac{\|\vec{f}(x)\|_{\ell^2}}{\sigma}\bigg)\cdot w(x)\,dx,
\end{equation*}
where the last inequality is due to $t\leq\Phi(t)=t\cdot(1+\log^+t)$ for any $t>0$. To deal with the other term $K_2$, let $Q_i^*=2\sqrt n Q_i$ be the cube concentric with $Q_i$ such that $\ell(Q_i^*)=(2\sqrt n)\ell(Q_i)$. Then we can further decompose $K_2$ as follows.
\begin{equation*}
\begin{split}
K_2\le&\,w\bigg(\bigg\{x\in\bigcup_i Q_i^*:\bigg(\sum_{j=1}^\infty\big|\big[b,\mathcal S_{\gamma}\big](h_j)(x)\big|^2\bigg)^{1/2}>\sigma/2\bigg\}\bigg)\\
&+w\bigg(\bigg\{x\notin \bigcup_i Q_i^*:\bigg(\sum_{j=1}^\infty\big|\big[b,\mathcal S_{\gamma}\big](h_j)(x)\big|^2\bigg)^{1/2}>\sigma/2\bigg\}\bigg)\\
:=&\,K_3+K_4.
\end{split}
\end{equation*}
Since $w\in A_1$, then by the inequality (\ref{weights}), we can get
\begin{equation*}
K_3\leq\sum_i w\big(Q_i^*\big)\le C\sum_i w(Q_i).
\end{equation*}
Furthermore, it follows from the inequality (\ref{decomposition}) and the $A_1$ condition that
\begin{equation*}
\begin{split}
K_3&\leq C\sum_i\frac{\,1\,}{\sigma}\cdot\underset{y\in Q_i}{\mbox{ess\,inf}}\,w(y)
\int_{Q_i}\bigg(\sum_{j=1}^\infty\big|f_j(y)\big|^2\bigg)^{1/2}dy\\
&\leq\frac{C}{\sigma}\sum_i\int_{Q_i}\bigg(\sum_{j=1}^\infty\big|f_j(y)\big|^2\bigg)^{1/2}w(y)\,dy
\leq\frac{C}{\sigma}\int_{\bigcup_i Q_i}\bigg(\sum_{j=1}^\infty\big|f_j(y)\big|^2\bigg)^{1/2}w(y)\,dy\\
&\leq C\int_{\mathbb R^n}\frac{\|\vec{f}(y)\|_{\ell^2}}{\sigma}\cdot w(y)\,dy
\leq C\int_{\mathbb R^n}\Phi\bigg(\frac{\|\vec{f}(y)\|_{\ell^2}}{\sigma}\bigg)\cdot w(y)\,dy,
\end{split}
\end{equation*}
where the last inequality is also due to $t\leq\Phi(t)$ for any $t>0$. Arguing as in the proof of Theorem \ref{mainthm:3}, for any given $x\in \mathbb R^n$, $(y,t)\in\Gamma(x)$ and for $j=1,2,\dots$, we also find that
\begin{align*}
&\sup_{\varphi\in{\mathcal C}_\gamma}\bigg|\int_{\mathbb R^n}\big[b(x)-b(z)\big]\varphi_t(y-z)\sum_i h_{ij}(z)\,dz\bigg|\notag\\
&\leq\sup_{\varphi\in{\mathcal C}_\gamma}\bigg|\sum_i \big[b(x)-b_{Q_i}\big]\int_{\mathbb R^n}\varphi_t(y-z)h_{ij}(z)\,dz\bigg|\notag\\
&+\sup_{\varphi\in{\mathcal C}_\gamma}\bigg|\int_{\mathbb R^n}\varphi_t(y-z)\sum_i \big[b_{Q_i}-b(z)\big]h_{ij}(z)\,dz\bigg|\\
&\leq\sum_i\big|b(x)-b_{Q_i}\big|\cdot\sup_{\varphi\in{\mathcal C}_\gamma}\bigg|\int_{\mathbb R^n}\varphi_t(y-z)h_{ij}(z)\,dz\bigg|\notag\\
&+\sup_{\varphi\in{\mathcal C}_\gamma}\bigg|\int_{\mathbb R^n}\varphi_t(y-z)\sum_i \big[b_{Q_i}-b(z)\big]h_{ij}(z)\,dz\bigg|.
\end{align*}
Hence, by definition, we have that for any given $x\in \mathbb R^n$ and $j\in\mathbb Z^+$,
\begin{equation}\label{commutator estimate}
\begin{split}
\big|\big[b,\mathcal S_\gamma\big](h_j)(x)\big|
&\leq\sum_i\big|b(x)-b_{Q_i}\big|\cdot\mathcal S_\gamma(h_{ij})(x)
+\mathcal S_\gamma\bigg(\sum_i[b_{Q_i}-b]h_{ij}\bigg)(x).
\end{split}
\end{equation}
On the other hand, by duality argument and Cauchy--Schwarz inequality, we can see the following vector-valued form of Minkowski's inequality is true for any real numbers $\nu_{ij}\in\mathbb R$, $i,j=1,2,\dots$.
\begin{equation}\label{min}
\bigg(\sum_{j}\bigg|\sum_i\big|\nu_{ij}\big|\bigg|^2\bigg)^{1/2}\leq\sum_i\bigg(\sum_{j}\big|\nu_{ij}\big|^2\bigg)^{1/2}.
\end{equation}
In view of the estimates (\ref{commutator estimate}) and (\ref{min}), we get
\begin{equation*}
\begin{split}
\bigg(\sum_{j=1}^\infty\big|\big[b,\mathcal S_{\gamma}\big](h_j)(x)\big|^2\bigg)^{1/2}
&\leq\bigg(\sum_{j=1}^\infty\bigg|\sum_i\big|b(x)-b_{Q_i}\big|\cdot\mathcal S_\gamma(h_{ij})(x)\bigg|^2\bigg)^{1/2}\\
&+\bigg(\sum_{j=1}^\infty\bigg|\mathcal S_\gamma\bigg(\sum_i[b_{Q_i}-b]h_{ij}\bigg)(x)\bigg|^2\bigg)^{1/2}\\
&\leq\sum_i\big|b(x)-b_{Q_i}\big|\cdot\bigg(\sum_{j=1}^\infty\big|\mathcal S_\gamma(h_{ij})(x)\big|^2\bigg)^{1/2}\\
&+\bigg(\sum_{j=1}^\infty\bigg|\mathcal S_\gamma\bigg(\sum_i[b_{Q_i}-b]h_{ij}\bigg)(x)\bigg|^2\bigg)^{1/2}.
\end{split}
\end{equation*}
Therefore, the term $K_4$ can be divided into two parts as follows:
\begin{equation*}
\begin{split}
K_4\leq &w\bigg(\bigg\{x\notin \bigcup_i Q_i^*:
\sum_i\big|b(x)-b_{Q_i}\big|\cdot\bigg(\sum_{j=1}^\infty\big|\mathcal S_\gamma(h_{ij})(x)\big|^2\bigg)^{1/2}>\sigma/4\bigg\}\bigg)\\
&+w\bigg(\bigg\{x\notin \bigcup_i Q_i^*:
\bigg(\sum_{j=1}^\infty\bigg|\mathcal S_\gamma\bigg(\sum_i[b_{Q_i}-b]h_{ij}\bigg)(x)\bigg|^2\bigg)^{1/2}>\sigma/4\bigg\}\bigg)\\
:=&K_5+K_6.
\end{split}
\end{equation*}
It follows directly from the Chebyshev's inequality that
\begin{equation*}
\begin{split}
K_5&\leq\frac{\,4\,}{\sigma}\int_{\mathbb R^n\backslash\bigcup_i Q_i^*}\sum_i\big|b(x)-b_{Q_i}\big|\cdot
\bigg(\sum_{j=1}^\infty\big|\mathcal S_\gamma(h_{ij})(x)\big|^2\bigg)^{1/2}w(x)\,dx\\
&\leq\frac{\,4\,}{\sigma}\sum_i\Bigg(\int_{(Q_i^*)^c}\big|b(x)-b_{Q_i}\big|\cdot
\bigg(\sum_{j=1}^\infty\big|\mathcal S_\gamma(h_{ij})(x)\big|^2\bigg)^{1/2}w(x)\,dx\Bigg).\\
\end{split}
\end{equation*}
Denote by $c_i$ the center of $Q_i$. For any $\varphi\in{\mathcal C}_\gamma$, $0<\gamma\le1$, by the cancellation condition of $h_{ij}$ over $Q_i$, we obtain that for any $(y,t)\in\Gamma(x)$ and for $i,j=1,2,\dots$,
\begin{align}\label{kernel estimate1}
\big|(\varphi_t*h_{ij})(y)\big|&=\left|\int_{Q_i}\big[\varphi_t(y-z)-\varphi_t(y-c_i)\big]h_{ij}(z)\,dz\right|\notag\\
&\leq\int_{Q_i\cap\{z:|z-y|\le t\}}\frac{|z-c_i|^\gamma}{t^{n+\gamma}}\big|h_{ij}(z)\big|\,dz\notag\\
&\leq C\cdot\frac{\ell(Q_i)^{\gamma}}{t^{n+\gamma}}\int_{Q_i\cap\{z:|z-y|\le t\}}\big|h_{ij}(z)\big|\,dz.
\end{align}
In addition, for any $z\in Q_i$ and $x\in (Q^*_i)^c$, we have $|z-c_i|<\frac{|x-c_i|}{2}$. Thus, for all $(y,t)\in\Gamma(x)$ and $|z-y|\le t$ with $z\in Q_i$, it is easy to see that
\begin{equation}\label{2t}
t+t\ge|x-y|+|y-z|\ge|x-z|\ge|x-c_i|-|z-c_i|\ge\frac{|x-c_i|}{2}.
\end{equation}
Hence, for any $x\in (Q^*_i)^c$, by using the above inequalities (\ref{kernel estimate1}) and (\ref{2t}) along with the fact that $ \big\|h_{ij}\big\|_{L^1}\leq 2\int_{Q_i}\big|f_j(x)\big|\,dx$, we obtain that for any $i,j=1,2,\dots$,
\begin{equation*}
\begin{split}
\big|\mathcal S_{\gamma}(h_{ij})(x)\big|&=\left(\iint_{\Gamma(x)}
\bigg[\sup_{\varphi\in{\mathcal C}_\gamma}\big|(\varphi_t*{h_{ij}})(y)\big|\bigg]^2\frac{dydt}{t^{n+1}}\right)^{1/2}\\
&\leq C\cdot\ell(Q_i)^{\gamma}\bigg(\int_{Q_i}\big|h_{ij}(z)\big|\,dz\bigg)
\bigg(\int_{\frac{|x-c_i|}{4}}^\infty\int_{|y-x|<t}\frac{dydt}{t^{2(n+\gamma)+n+1}}\bigg)^{1/2}\\
&\leq C\cdot\ell(Q_i)^{\gamma}\bigg(\int_{Q_i}\big|h_{ij}(z)\big|\,dz\bigg)
\bigg(\int_{\frac{|x-c_i|}{4}}^\infty\frac{dt}{t^{2(n+\gamma)+1}}\bigg)^{1/2}\\
&\leq C\cdot\frac{\ell(Q_i)^{\gamma}}{|x-c_i|^{n+\gamma}}\bigg(\int_{Q_i}\big|f_j(z)\big|\,dz\bigg).
\end{split}
\end{equation*}
Furthermore, by duality and Cauchy--Schwarz inequality again, one has
\begin{equation*}
\bigg(\sum_{j=1}^\infty\big|\mathcal S_\gamma(h_{ij})(x)\big|^2\bigg)^{1/2}\leq
C\cdot\frac{\ell(Q_i)^{\gamma}}{|x-c_i|^{n+\gamma}}\times\int_{Q_i}\bigg(\sum_{j=1}^\infty\big|f_j(z)\big|^2\bigg)^{1/2}dz.
\end{equation*}
Since $Q_i^*=2\sqrt n Q_i\supset 2Q_i$, then $(Q_i^*)^c\subset (2Q_i)^c$. This fact together with the pointwise estimate derived above yields
\begin{equation*}
\begin{split}
K_5&\leq\frac{C}{\sigma}\sum_i
\left(\ell(Q_i)^{\gamma}\int_{Q_i}\bigg(\sum_{j=1}^\infty\big|f_j(z)\big|^2\bigg)^{1/2}dz
\times\int_{(Q_i^*)^c}\big|b(x)-b_{Q_i}\big|\cdot\frac{w(x)}{|x-c_i|^{n+\gamma}}dx\right)\\
&\leq\frac{C}{\sigma}\sum_i
\left(\ell(Q_i)^{\gamma}\int_{Q_i}\bigg(\sum_{j=1}^\infty\big|f_j(z)\big|^2\bigg)^{1/2}dz
\times\int_{(2Q_i)^c}\big|b(x)-b_{Q_i}\big|\cdot\frac{w(x)}{|x-c_i|^{n+\gamma}}dx\right)\\
\end{split}
\end{equation*}
\begin{equation*}
\begin{split}
&\leq\frac{C}{\sigma}\sum_i
\left(\ell(Q_i)^{\gamma}\int_{Q_i}\bigg(\sum_{j=1}^\infty\big|f_j(z)\big|^2\bigg)^{1/2}dz
\times\sum_{l=1}^\infty\int_{2^{l+1}Q_i\backslash 2^{l}Q_i}\big|b(x)-b_{2^{l+1}Q_i}\big|\cdot\frac{w(x)}{|x-c_i|^{n+\gamma}}dx\right)\\
&+\frac{C}{\sigma}\sum_i
\left(\ell(Q_i)^{\gamma}\int_{Q_i}\bigg(\sum_{j=1}^\infty\big|f_j(z)\big|^2\bigg)^{1/2}dz\times\sum_{l=1}^\infty\int_{2^{l+1}Q_i\backslash 2^{l}Q_i}\big|b_{2^{l+1}Q_i}-b_{Q_i}\big|\cdot\frac{w(x)}{|x-c_i|^{n+\gamma}}dx\right)\\
&:=\mbox{\upshape I+II}.
\end{split}
\end{equation*}
For the term I, it then follows from $(ii)$ of Lemma \ref{BMO}(consider $2^{l+1}Q_i$ instead of $B$), (\ref{general weights}) and the assumption $w\in A_1$ that
\begin{equation*}
\begin{split}
\mbox{\upshape I}&\leq\frac{C}{\sigma}\sum_i
\left(\ell(Q_i)^{\gamma}\int_{Q_i}\bigg(\sum_{j=1}^\infty\big|f_j(z)\big|^2\bigg)^{1/2}dz\times\sum_{l=1}^\infty\frac{1}{[2^{l-1}\ell(Q_i)]^{n+\gamma}}
\int_{2^{l+1}Q_i}\big|b(x)-b_{2^{l+1}Q_i}\big|\cdot w(x)\,dx\right)\\
&\leq\frac{C\cdot\|b\|_*}{\sigma}\sum_i\left(\int_{Q_i}\bigg(\sum_{j=1}^\infty\big|f_j(z)\big|^2\bigg)^{1/2}dz
\times\sum_{l=1}^\infty\frac{w\big(2^{l+1}Q_i\big)}{(2^{l-1})^{n+\gamma}|Q_i|}\right)\\
&\leq\frac{C\cdot\|b\|_*}{\sigma}\sum_i\left(\int_{Q_i}\bigg(\sum_{j=1}^\infty\big|f_j(z)\big|^2\bigg)^{1/2}dz
\times\sum_{l=1}^\infty\frac{(2^{l+1})^nw\big(Q_i\big)}{(2^{l-1})^{n+\gamma}|Q_i|}\right)\\
&\leq\frac{C}{\sigma}\sum_i\left(\frac{w\big(Q_i\big)}{|Q_i|}\cdot\int_{Q_i}\bigg(\sum_{j=1}^\infty\big|f_j(z)\big|^2\bigg)^{1/2}dz
\times\sum_{l=1}^\infty\frac{1}{2^{l\gamma}}\right)\\
&\leq\frac{C}{\sigma}\sum_i\underset{z\in Q_i}{\mbox{ess\,inf}}\,w(z)\int_{Q_i}\bigg(\sum_{j=1}^\infty\big|f_j(z)\big|^2\bigg)^{1/2}dz
\leq \frac{C}{\sigma}\int_{\bigcup_i Q_i}\bigg(\sum_{j=1}^\infty\big|f_j(z)\big|^2\bigg)^{1/2}w(z)\,dz\\
&\leq C\int_{\mathbb R^n}\frac{\|\vec{f}(z)\|_{\ell^2}}{\sigma}\cdot w(z)\,dz
\leq C\int_{\mathbb R^n}\Phi\bigg(\frac{\|\vec{f}(z)\|_{\ell^2}}{\sigma}\bigg)\cdot w(z)\,dz.
\end{split}
\end{equation*}
For the term II, from $(i)$ of Lemma \ref{BMO} and (\ref{general weights}) along with the assumption $w\in A_1$, it then follows that
\begin{equation*}
\begin{split}
\mbox{\upshape II}&\leq\frac{C\cdot\|b\|_*}{\sigma}\sum_i\left(\ell(Q_i)^{\gamma}\int_{Q_i}\bigg(\sum_{j=1}^\infty\big|f_j(z)\big|^2\bigg)^{1/2}dz
\times\sum_{l=1}^\infty\big(l+1\big)\cdot\frac{w\big(2^{l+1}Q_i\big)}{[2^{l-1}\ell(Q_i)]^{n+\gamma}}\right)\\
&\leq\frac{C\cdot\|b\|_*}{\sigma}\sum_i\left(\int_{Q_i}\bigg(\sum_{j=1}^\infty\big|f_j(z)\big|^2\bigg)^{1/2}dz
\times\sum_{l=1}^\infty\big(l+1\big)\cdot\frac{(2^{l+1})^nw\big(Q_i\big)}{(2^{l-1})^{n+\gamma}|Q_i|}\right)\\
&\leq\frac{C}{\sigma}\sum_i\left(\frac{w\big(Q_i\big)}{|Q_i|}\cdot\int_{Q_i}\bigg(\sum_{j=1}^\infty\big|f_j(z)\big|^2\bigg)^{1/2}dz
\times\sum_{l=1}^\infty\frac{l+1}{2^{l\gamma}}\right)\\
&\leq\frac{C}{\sigma}\sum_i\left(\frac{w\big(Q_i\big)}{|Q_i|}\cdot\int_{Q_i}\bigg(\sum_{j=1}^\infty\big|f_j(z)\big|^2\bigg)^{1/2}dz\right)
\leq C\int_{\mathbb R^n}\Phi\bigg(\frac{\|\vec{f}(z)\|_{\ell^2}}{\sigma}\bigg)\cdot w(z)\,dz.
\end{split}
\end{equation*}
On the other hand, by using the weighted weak-type (1,1) estimate of vector-valued intrinsic square function (see Theorem \ref{weak}) and (\ref{min}), we have
\begin{equation*}
\begin{split}
K_6&\leq\frac{C}{\sigma}\int_{\mathbb R^n}\bigg(\sum_{j=1}^\infty\bigg|\sum_i\big|b(x)-b_{Q_i}\big|\big|h_{ij}(x)\big|\bigg|^2\bigg)^{1/2}w(x)\,dx\\
&\leq\frac{C}{\sigma}\int_{\mathbb R^n}\sum_i\big|b(x)-b_{Q_i}\big|\bigg(\sum_{j=1}^\infty\big|h_{ij}(x)\big|^2\bigg)^{1/2}w(x)\,dx\\
&=\frac{C}{\sigma}\sum_i\int_{Q_i}\big|b(x)-b_{Q_i}\big| \bigg(\sum_{j=1}^\infty\big|h_{j}(x)\big|^2\bigg)^{1/2}w(x)\,dx\\
&\leq\frac{C}{\sigma}\sum_i\int_{Q_i}\big|b(x)-b_{Q_i}\big|\bigg(\sum_{j=1}^\infty\big|f_j(x)\big|^2\bigg)^{1/2}w(x)\,dx\\
&+\frac{C}{\sigma}\sum_i\frac{1}{|Q_i|}\int_{Q_i}\bigg(\sum_{j=1}^\infty\big|f_j(y)\big|^2\bigg)^{1/2}dy\times\int_{Q_i}\big|b(x)-b_{Q_i}\big|w(x)\,dx\\
&:=\mbox{\upshape III+IV}.
\end{split}
\end{equation*}
For the term III, by the generalized H\"older's inequality (\ref{Wholder}), we can deduce that
\begin{equation*}
\begin{split}
\mbox{\upshape III}=&\frac{C}{\sigma}\sum_i w(Q_i)\cdot\frac{1}{w(Q_i)}\int_{Q_i}\big|b(x)-b_{Q_i}\big|\bigg(\sum_{j=1}^\infty\big|f_j(x)\big|^2\bigg)^{1/2}w(x)\,dx\\
\leq&\frac{C}{\sigma}\sum_i w(Q_i)\cdot\big\|b-b_{Q_i}\big\|_{\exp L(w),Q_i}
\bigg\|\bigg(\sum_{j=1}^\infty\big|f_j\big|^2\bigg)^{1/2}\bigg\|_{L\log L(w),Q_i}\\
\leq&\frac{C\cdot\|b\|_*}{\sigma}\sum_i w(Q_i)\cdot
\bigg\|\bigg(\sum_{j=1}^\infty\big|f_j\big|^2\bigg)^{1/2}\bigg\|_{L\log L(w),Q_i}.\\
\end{split}
\end{equation*}
In the last inequality, we have used the well-known fact that (see \cite{zhang})
\begin{equation}\label{Jensen}
\big\|b-b_{Q}\big\|_{\exp L(w),Q}\leq C\|b\|_*,\qquad \mbox{for any cube }Q\subset\mathbb R^n.
\end{equation}
It is equivalent to the inequality
\begin{equation*}
\frac{1}{w(Q)}\int_Q\exp\bigg(\frac{|b(y)-b_Q|}{c_0\|b\|_*}\bigg)w(y)\,dy\leq C,
\end{equation*}
which is just a corollary of the well-known John--Nirenberg's inequality (see \cite{john}) and the comparison property of $A_1$ weights. Moreover, it can be shown that for every cube $Q\subset\mathbb R^n$ and $w\in A_1$ (see \cite{rao,zhang}),
\begin{equation}\label{equiv norm with weight}
\big\|f\big\|_{L\log L(w),Q}\approx \inf_{\eta>0}\left\{\eta+\frac{\eta}{w(Q)}\int_Q\Phi\left(\frac{|f(y)|}{\eta}\right)\cdot w(y)\,dy\right\}.
\end{equation}
Hence, it is concluded from \eqref{equiv norm with weight} that
\begin{equation*}
\begin{split}
\mbox{\upshape III}\leq&\frac{C\cdot\|b\|_*}{\sigma}\sum_i w(Q_i)\cdot
\inf_{\eta>0}\left\{\eta+\frac{\eta}{w(Q_i)}\int_{Q_i}\Phi\bigg(\frac{\|\vec{f}(y)\|_{\ell^2}}{\eta}\bigg)\cdot w(y)\,dy\right\}\\
\leq&\frac{C\cdot\|b\|_*}{\sigma}\sum_i w(Q_i)\cdot
\left\{\sigma+\frac{\sigma}{w(Q_i)}\int_{Q_i}\Phi\bigg(\frac{\|\vec{f}(y)\|_{\ell^2}}{\sigma}\bigg)\cdot w(y)\,dy\right\}\\
\leq& C\left\{\sum_i w(Q_i)+\sum_i\int_{Q_i}\Phi\bigg(\frac{\|\vec{f}(y)\|_{\ell^2}}{\sigma}\bigg)\cdot w(y)\,dy\right\}\\
\leq& C\int_{\mathbb R^n}\Phi\bigg(\frac{\|\vec{f}(y)\|_{\ell^2}}{\sigma}\bigg)\cdot w(y)\,dy.
\end{split}
\end{equation*}
For the term IV, by $(ii)$ of Lemma \ref{BMO}(consider $Q_i$ instead of $B$) and the assumption $w\in A_1$, we conclude that
\begin{equation*}
\begin{split}
\mbox{\upshape IV}\leq&\frac{C\cdot\|b\|_*}{\sigma}\sum_i\frac{w(Q_i)}{|Q_i|}\int_{Q_i}\bigg(\sum_{j=1}^\infty\big|f_j(y)\big|^2\bigg)^{1/2}dy\\
\leq&\frac{C\cdot\|b\|_*}{\sigma}\sum_i\int_{Q_i}\bigg(\sum_{j=1}^\infty\big|f_j(y)\big|^2\bigg)^{1/2}w(y)\,dy\\
\leq&\frac{C\cdot\|b\|_*}{\sigma}\int_{\bigcup_i Q_i}\bigg(\sum_{j=1}^\infty\big|f_j(y)\big|^2\bigg)^{1/2}w(y)\,dy\\
\leq& C\int_{\mathbb R^n}\frac{\|\vec{f}(y)\|_{\ell^2}}{\sigma}\cdot w(y)\,dy
\leq C\int_{\mathbb R^n}\Phi\bigg(\frac{\|\vec{f}(y)\|_{\ell^2}}{\sigma}\bigg)\cdot w(y)\,dy.
\end{split}
\end{equation*}
Summing up all the above estimates, we get the desired inequality \eqref{end}.
\end{proof}

\begin{proof}[Proof of Theorem $\ref{mainthm:4}$]
For any fixed ball $B=B(y,r)$ in $\mathbb R^n$, as before, we represent $f_j$ as $f_j=f^0_j+f^\infty_j$, where $f^0_j=f_j\cdot\chi_{2B}$ and $f^\infty_j=f_j\cdot\chi_{(2B)^c}$, $j=1,2,\ldots$. Then for any given $\sigma>0$, one can write
\begin{align}\label{Jprime}
&w(B(y,r))^{1/{\alpha}-1-1/q}\cdot
w\bigg(\bigg\{x\in B(y,r):\bigg(\sum_{j=1}^\infty\big|\big[b,\mathcal S_{\gamma}\big](f_j)(x)\big|^2\bigg)^{1/2}>\sigma\bigg\}\bigg)\notag\\
\leq &w(B(y,r))^{1/{\alpha}-1-1/q}\cdot
w\bigg(\bigg\{x\in B(y,r):\bigg(\sum_{j=1}^\infty\big|\big[b,\mathcal S_{\gamma}\big](f^0_j)(x)\big|^2\bigg)^{1/2}>\sigma/2\bigg\}\bigg)\notag\\
&+w(B(y,r))^{1/{\alpha}-1-1/q}\cdot
w\bigg(\bigg\{x\in B(y,r):\bigg(\sum_{j=1}^\infty\big|\big[b,\mathcal S_{\gamma}\big](f^\infty_j)(x)\big|^2\bigg)^{1/2}>\sigma/2\bigg\}\bigg)\notag\\
:=&J'_1(y,r)+J'_2(y,r).
\end{align}
By using Theorem \ref{mainthm:5}, we obtain
\begin{equation*}
\begin{split}
J'_1(y,r)&\leq C\cdot w(B(y,r))^{1/{\alpha}-1-1/q}\int_{\mathbb R^n}\Phi\bigg(\frac{\|\vec{f}_0(x)\|_{\ell^2}}{\sigma}\bigg)\cdot w(x)\,dx\\
&= C\cdot w(B(y,r))^{1/{\alpha}-1-1/q}\int_{B(y,2r)}\Phi\bigg(\frac{\|\vec{f}(x)\|_{\ell^2}}{\sigma}\bigg)\cdot w(x)\,dx\\
&= C\cdot\frac{w(B(y,r))^{1/{\alpha}-1-1/q}}{w(B(y,2r))^{1/{\alpha}-1-1/q}}\cdot\frac{w(B(y,2r))^{1/{\alpha}-1/q}}{w(B(y,2r))}
\int_{B(y,2r)}\Phi\bigg(\frac{\|\vec{f}(x)\|_{\ell^2}}{\sigma}\bigg)\cdot w(x)\,dx.
\end{split}
\end{equation*}
Here we use the following notation:
\begin{equation*}
\big\|\vec{f}_0(x)\big\|_{\ell^2}:=\bigg(\sum_{j=1}^\infty|f^0_j(x)|^2\bigg)^{1/2}=\bigg(\sum_{j=1}^\infty|f_j(x)\cdot\chi_{2B}|^2\bigg)^{1/2}.
\end{equation*}
Moreover, since $w\in A_1$, then by inequalities \eqref{doubling2} and \eqref{main esti1}, we have
\begin{equation}\label{WJ1yr}
\begin{split}
J'_1(y,r)&\leq C\cdot\frac{w(B(y,2r))^{1/{\alpha}-1/q}}{w(B(y,2r))}
\int_{B(y,2r)}\Phi\bigg(\frac{\|\vec{f}(x)\|_{\ell^2}}{\sigma}\bigg)\cdot w(x)\,dx\\
&\leq C\cdot w(B(y,2r))^{1/{\alpha}-1/q}\bigg\|\Phi\bigg(\frac{\|\vec{f}(\cdot)\|_{\ell^2}}{\sigma}\bigg)\bigg\|_{L\log L(w),B(y,2r)}.
\end{split}
\end{equation}
We now turn to deal with the term $J'_2(y,r)$. Recall that the following inequality
\begin{equation*}
\begin{split}
\bigg(\sum_{j=1}^\infty\big|\big[b,\mathcal S_\gamma\big](f^\infty_j)(x)\big|^2\bigg)^{1/2}
&\leq\big|b(x)-b_{B(y,r)}\big|\bigg(\sum_{j=1}^\infty\big|\mathcal S_\gamma(f^{\infty}_j)(x)\big|^2\bigg)^{1/2}\\
&+\bigg(\sum_{j=1}^\infty\Big|\mathcal S_\gamma\Big([b_{B(y,r)}-b]f^\infty_j\Big)(x)\Big|^2\bigg)^{1/2}.
\end{split}
\end{equation*}
is valid. Thus, we can further decompose $J'_2(y,r)$ as
\begin{equation*}
\begin{split}
J'_2(y,r)\leq&w(B(y,r))^{1/{\alpha}-1-1/q}\cdot
w\bigg(\bigg\{x\in B(y,r):\big|b(x)-b_{B(y,r)}\big|\cdot
\bigg(\sum_{j=1}^\infty\big|\mathcal S_\gamma(f^{\infty}_j)(x)\big|^2\bigg)^{1/2}>\sigma/4\bigg\}\bigg)\\
&+w(B(y,r))^{1/{\alpha}-1-1/q}\cdot
w\bigg(\bigg\{x\in B(y,r):\bigg(\sum_{j=1}^\infty\Big|\mathcal S_\gamma\Big([b_{B(y,r)}-b]f^\infty_j\Big)(x)\Big|^2\bigg)^{1/2}>\sigma/4\bigg\}\bigg)\\
:=&J'_3(y,r)+J'_4(y,r).
\end{split}
\end{equation*}
Applying Chebyshev's inequality and previous pointwise estimate \eqref{key estimate1}, together with $(ii)$ of Lemma \ref{BMO}, we deduce that
\begin{equation*}
\begin{split}
J'_3(y,r)&\leq w(B(y,r))^{1/{\alpha}-1-1/q}\cdot\frac{\,4\,}{\sigma}
\int_{B(y,r)}\big|b(x)-b_{B(y,r)}\big|\cdot\bigg(\sum_{j=1}^\infty\big|\mathcal S_\gamma(f^{\infty}_j)(x)\big|^2\bigg)^{1/2}w(x)\,dx\\
&\leq C\cdot w(B(y,r))^{1/{\alpha}-1/q}
\sum_{l=1}^\infty\frac{1}{|B(y,2^{l+1}r)|}\int_{B(y,2^{l+1}r)}\frac{\|\vec{f}(z)\|_{\ell^2}}{\sigma}dz\\
&\times\frac{1}{w(B(y,r))}\int_{B(y,r)}\big|b(x)-b_{B(y,r)}\big|w(x)\,dx\\
&\leq C\|b\|_*\sum_{l=1}^\infty\frac{1}{|B(y,2^{l+1}r)|}\int_{B(y,2^{l+1}r)}\frac{\|\vec{f}(z)\|_{\ell^2}}{\sigma}dz
\times w(B(y,r))^{1/{\alpha}-1/q}.
\end{split}
\end{equation*}
Furthermore, note that $t\leq\Phi(t)=t\cdot(1+\log^+t)$ for any $t>0$. This fact, together with previous estimate \eqref{main esti1} and the $A_1$ condition, implies that
\begin{equation*}
\begin{split}
J'_3(y,r)&\leq C\|b\|_*\sum_{l=1}^\infty\frac{1}{w(B(y,2^{l+1}r))}\int_{B(y,2^{l+1}r)}\frac{\|\vec{f}(z)\|_{\ell^2}}{\sigma}\cdot w(z)\,dz
\times w(B(y,r))^{1/{\alpha}-1/q}\\
&\leq C\|b\|_*\sum_{l=1}^\infty\frac{1}{w(B(y,2^{l+1}r))}\int_{B(y,2^{l+1}r)}\Phi\bigg(\frac{\|\vec{f}(z)\|_{\ell^2}}{\sigma}\bigg)\cdot w(z)\,dz
\times w(B(y,r))^{1/{\alpha}-1/q}\\
&\leq C\|b\|_*\sum_{l=1}^\infty\bigg\|\Phi\bigg(\frac{\|\vec{f}(\cdot)\|_{\ell^2}}{\sigma}\bigg)\bigg\|_{L\log L(w),B(y,2^{l+1}r)}
\times w(B(y,r))^{1/{\alpha}-1/q}.
\end{split}
\end{equation*}
On the other hand, applying the pointwise estimate \eqref{key estimate2} and Chebyshev's inequality, we have
\begin{equation*}
\begin{split}
J'_4(y,r)&\leq w(B(y,r))^{1/{\alpha}-1-1/q}\cdot\frac{\,4\,}{\sigma}
\int_{B(y,r)}\bigg(\sum_{j=1}^\infty\Big|\mathcal S_\gamma\Big([b_{B(y,r)}-b]f^\infty_j\Big)(x)\Big|^2\bigg)^{1/2}w(x)\,dx\\
&\leq w(B(y,r))^{1/{\alpha}-1/q}\cdot\frac{\,C\,}{\sigma}
\sum_{l=1}^\infty\frac{1}{|B(y,2^{l+1}r)|}\int_{B(y,2^{l+1}r)}\big|b(z)-b_{B(y,r)}\big|\cdot\bigg(\sum_{j=1}^\infty\big|f_j(z)\big|^2\bigg)^{1/2}dz\\
&\leq w(B(y,r))^{1/{\alpha}-1/q}\cdot\frac{\,C\,}{\sigma}
\sum_{l=1}^\infty\frac{1}{|B(y,2^{l+1}r)|}\int_{B(y,2^{l+1}r)}\big|b(z)-b_{B(y,2^{l+1}r)}\big|
\cdot\bigg(\sum_{j=1}^\infty\big|f_j(z)\big|^2\bigg)^{1/2}dz\\
&+w(B(y,r))^{1/{\alpha}-1/q}\cdot\frac{\,C\,}{\sigma}
\sum_{l=1}^\infty\frac{1}{|B(y,2^{l+1}r)|}\int_{B(y,2^{l+1}r)}\big|b_{B(y,2^{l+1}r)}-b_{B(y,r)}\big|
\cdot\bigg(\sum_{j=1}^\infty\big|f_j(z)\big|^2\bigg)^{1/2}dz\\
&:=J'_5(y,r)+J'_6(y,r).
\end{split}
\end{equation*}
For the term $J'_5(y,r)$, since $w\in A_1$, it follows from the $A_1$ condition and the inequality $t\leq \Phi(t)$ that
\begin{equation*}
\begin{split}
J'_5(y,r)&\leq w(B(y,r))^{1/{\alpha}-1/q}\\
&\times\frac{\,C\,}{\sigma} \sum_{l=1}^\infty\frac{1}{w(B(y,2^{l+1}r))}\int_{B(y,2^{l+1}r)}\big|b(z)-b_{B(y,2^{l+1}r)}\big|
\cdot\bigg(\sum_{j=1}^\infty\big|f_j(z)\big|^2\bigg)^{1/2}w(z)\,dz\\
&\leq C\cdot w(B(y,r))^{1/{\alpha}-1/q}\\
&\times\sum_{l=1}^\infty\frac{1}{w(B(y,2^{l+1}r))}\int_{B(y,2^{l+1}r)}\big|b(z)-b_{B(y,2^{l+1}r)}\big|
\cdot\Phi\bigg(\frac{\|\vec{f}(z)\|_{\ell^2}}{\sigma}\bigg)w(z)\,dz.
\end{split}
\end{equation*}
Furthermore, we use the generalized H\"older's inequality \eqref{Wholder} and \eqref{Jensen} to obtain
\begin{equation*}
\begin{split}
J'_5(y,r)&\leq C\cdot w(B(y,r))^{1/{\alpha}-1/q}\\
&\times\sum_{l=1}^\infty\big\|b-b_{B(y,2^{l+1}r)}\big\|_{\exp L(w),B(y,2^{l+1}r)}
\bigg\|\Phi\bigg(\frac{\|\vec{f}(\cdot)\|_{\ell^2}}{\sigma}\bigg)\bigg\|_{L\log L(w),B(y,2^{l+1}r)}\\
&\leq C\|b\|_*\sum_{l=1}^\infty\bigg\|\Phi\bigg(\frac{\|\vec{f}(\cdot)\|_{\ell^2}}{\sigma}\bigg)\bigg\|_{L\log L(w),B(y,2^{l+1}r)}
\times w(B(y,r))^{1/{\alpha}-1/q}.
\end{split}
\end{equation*}
For the last term $J'_6(y,r)$ we proceed as follows. Using $(i)$ of Lemma \ref{BMO} together with the $A_1$ condition on $w$ and the inequality $t\leq\Phi(t)$, we deduce that
\begin{equation*}
\begin{split}
J'_6(y,r)&\leq C\cdot w(B(y,r))^{1/{\alpha}-1/q}
\sum_{l=1}^\infty(l+1)\|b\|_*\cdot\frac{1}{|B(y,2^{l+1}r)|}\int_{B(y,2^{l+1}r)}\frac{\|\vec{f}(z)\|_{\ell^2}}{\sigma}dz\\
&\leq C\cdot w(B(y,r))^{1/{\alpha}-1/q}
\sum_{l=1}^\infty(l+1)\|b\|_*\cdot\frac{1}{w(B(y,2^{l+1}r))}\int_{B(y,2^{l+1}r)}\frac{\|\vec{f}(z)\|_{\ell^2}}{\sigma}\cdot w(z)\,dz\\
&\leq C\|b\|_*\cdot w(B(y,r))^{1/{\alpha}-1/q}
\sum_{l=1}^\infty\frac{(l+1)}{w(B(y,2^{l+1}r))}\int_{B(y,2^{l+1}r)}\Phi\bigg(\frac{\|\vec{f}(z)\|_{\ell^2}}{\sigma}\bigg)\cdot w(z)\,dz\\
&\leq C\|b\|_*\sum_{l=1}^\infty\big(l+1\big)\cdot\bigg\|\Phi\bigg(\frac{\|\vec{f}(\cdot)\|_{\ell^2}}{\sigma}\bigg)\bigg\|_{L\log L(w),B(y,2^{l+1}r)}
\times w(B(y,r))^{1/{\alpha}-1/q},
\end{split}
\end{equation*}
where in the last inequality we have used \eqref{main esti1}. Summarizing the estimates derived above, we conclude that
\begin{align}\label{WJ2yr}
J'_2(y,r)&\leq C\|b\|_* 
\sum_{l=1}^\infty\big(l+1\big)\cdot\bigg\|\Phi\bigg(\frac{\|\vec{f}(\cdot)\|_{\ell^2}}{\sigma}\bigg)\bigg\|_{L\log L(w),B(y,2^{l+1}r)}
\times w(B(y,r))^{1/{\alpha}-1/q}\notag\\
&=C\sum_{l=1}^\infty w(B(y,2^{l+1}r))^{1/{\alpha}-1/q}
\bigg\|\Phi\bigg(\frac{\|\vec{f}(\cdot)\|_{\ell^2}}{\sigma}\bigg)\bigg\|_{L\log L(w),B(y,2^{l+1}r)}\notag\\
&\times\big(l+1\big)\cdot\frac{w(B(y,r))^{1/{\alpha}-1/q}}{w(B(y,2^{l+1}r))^{1/{\alpha}-1/q}}.
\end{align}
Therefore by taking the $L^q_{\mu}$-norm of both sides of \eqref{Jprime}(with respect to the variable $y$), and then using Minkowski's inequality, \eqref{WJ1yr} and \eqref{WJ2yr}, we finally obtain
\begin{equation*}
\begin{split}
&\Bigg\|w(B(y,r))^{1/{\alpha}-1-1/q}\cdot
w\bigg(\bigg\{x\in B(y,r):\bigg(\sum_{j=1}^\infty\big|\big[b,\mathcal S_{\gamma}\big](f_j)(x)\big|^2\bigg)^{1/2}>\sigma\bigg\}\bigg)\Bigg\|_{L^q_{\mu}}\\
&\leq\big\|J'_1(y,r)\big\|_{L^q_{\mu}}+\big\|J'_2(y,r)\big\|_{L^q_{\mu}}\\
&\leq C\bigg\|w(B(y,2r))^{1/{\alpha}-1/q}\bigg\|\Phi\bigg(\frac{\|\vec{f}(\cdot)\|_{\ell^2}}{\sigma}\bigg)\bigg\|_{L\log L(w),B(y,2r)}\bigg\|_{L^q_{\mu}}\\
&+C\sum_{l=1}^\infty\bigg\|w(B(y,2^{l+1}r))^{1/{\alpha}-1/q}
\bigg\|\Phi\bigg(\frac{\|\vec{f}(\cdot)\|_{\ell^2}}{\sigma}\bigg)\bigg\|_{L\log L(w),B(y,2^{l+1}r)}\bigg\|_{L^q_{\mu}}\\
&\times\big(l+1\big)\cdot\frac{w(B(y,r))^{1/{\alpha}-1/q}}{w(B(y,2^{l+1}r))^{1/{\alpha}-1/q}}\\
&\leq C\bigg\|\Phi\bigg(\frac{\|\vec{f}(\cdot)\|_{\ell^2}}{\sigma}\bigg)\bigg\|_{(L\log L,L^q)^{\alpha}(w;\mu)}\\
&+C\bigg\|\Phi\bigg(\frac{\|\vec{f}(\cdot)\|_{\ell^2}}{\sigma}\bigg)\bigg\|_{(L\log L,L^q)^{\alpha}(w;\mu)}
\times\sum_{l=1}^\infty\big(l+1\big)\cdot\frac{w(B(y,r))^{1/{\alpha}-1/q}}{w(B(y,2^{l+1}r))^{1/{\alpha}-1/q}}\\
&\leq C\bigg\|\Phi\bigg(\frac{\|\vec{f}(\cdot)\|_{\ell^2}}{\sigma}\bigg)\bigg\|_{(L\log L,L^q)^{\alpha}(w;\mu)},
\end{split}
\end{equation*}
where the last inequality follows from \eqref{psi3}. This completes the proof of Theorem \ref{mainthm:4}.
\end{proof}

\end{document}